%% file: main.tex
\documentclass{m2an}
\usepackage{cases}

\def\sinh{\operatorname{sinh}}

\def\vv{\mathbf v}

\begin{document}
\title{Two mixed finite element formulations for the weak imposition of the Neumann boundary conditions for the Darcy flow}
\author{Erik Burman}\address{Chair of Computational Mathematics, University College London, London, UK. E-mail address: e.burman@ucl.ac.uk}
\author{Riccardo Puppi}\address{Chair of Modelling and Numerical Simulation, \'Ecole Polytechnique F\'ed\'erale de Lausanne, Lausanne, CH. E-mail address: riccardo.puppi@epfl.ch}
\date{April 2, 2021}  	

\thanks{Erik Burman was partially supported by the EPSRC grants EP/P01576X/1 and EP/T033126/1. Riccardo Puppi was partially supported by ERC AdG project CHANGE n. 694515.}
\begin{abstract} 
We propose two different discrete formulations for the weak imposition of the Neumann boundary conditions of the Darcy flow. The Raviart-Thomas mixed finite element on both triangular and quadrilateral meshes is considered for both methods. One is a consistent discretization depending on a weighting parameter scaling as $\mathcal O(h^{-1})$, while the other is a penalty-type formulation obtained as the discretization of a perturbation of the original problem and relies on a parameter scaling as $\mathcal O(h^{-k-1})$, $k$ being the order of the Raviart-Thomas space. We rigorously prove that both methods are stable and result in optimal convergent numerical schemes with respect to appropriate mesh-dependent norms, although the chosen norms do not scale as the usual $L^2$-norm. However, we are still able to recover the optimal a priori $L^2$-error estimates for the velocity field, respectively, for high-order and the lowest-order Raviart-Thomas discretizations, for the first and second numerical schemes. Finally, some numerical examples validating the theory are exhibited.
 \end{abstract}
%
%
\maketitle
\section*{Introduction}
We consider the finite element approximation for the weak imposition of the Neumann boundary conditions for the Poisson problem in its mixed formulation, also known as Darcy's law in the context of fluid dynamics. 

Let us point out that the situation is dual with respect to the standard formulation of the Poisson problem: here the Neumann boundary conditions are essential and to the best of our knowledge it is not clear in the literature how to proceed in order to enforce them by manipulating the weak formulation rather than the functional spaces.

As far as the primary formulation is concerned, a wide variety of techniques have already been proposed and are now well-understood, the most prominent of which are undoubtedly the penalty method introduced in~\cite{babuska_penalty}, the Lagrange multipliers approach of~\cite{babuska_lagrange} and, of course, the Nitsche method developed in~\cite{nitsche} and later promoted in~\cite{STENBERG1995139}, where the author relates it to the stabilized Lagrange multiplier method of~\cite{barbosa_hughes}. 

In this work two different discrete formulations of the Darcy problem for the weak imposition of the Neumann boundary conditions are provided.
Both of them are based on the Raviart-Thomas finite element discretization for triangular and quadrilateral meshes. Let us notice that the two schemes do not add any additional degrees of freedoms and moreover, for simplicity, all dimensionless parameters have been set to $1$.

The first formulation is a consistent discretization of the Darcy system, hence it falls into the class of Nitsche-type methods. A weighting parameter scaling as $\mathcal O(h^{-1})$ needs to be introduced. Let us observe that this formulation had already appeared in the literature in~\cite{dangelo_scotti} for the lowest-order Raviart-Thomas element, in the context of incompressible flows in fractured media.

The latter, which is inspired by~\cite{konno_robin} and based on a perturbed variational principle, belongs instead to the family of penalty methods. In this case the penalization parameter scales as $\mathcal O(h^{-k-1})$, $k$ being the order of the Raviart-Thomas discretization, entailing a much more severe ill-conditioning of the resulting stiffness matrix.

We are able to prove that both formulations are stable and give rise to optimal convergent schemes with respect to suitable mesh-dependent norms, which do not scale as the $L^2$-norm as it is customary for the Darcy problem. At this point we are able to demonstrate super convergence results that allow us to find an optimal a priori estimate of the velocity error with respect to the $L^2$-norm, respectively, for any higher order Raviart-Thomas discretization combined with the first method and for the lowest order element with the second formulation.

Note that this work should be considered as a preliminary step towards the much more involved situation of an underlying mesh which is not fitted with the boundary of the physical domain~\cite{NME:NME4823}.


Let us briefly sketch the outline of the paper. In the next two sections we introduce, respectively, the strong formulation of the Darcy problem and a singularly perturbed formulation of it, parametrized by $\eps\ge 0^{+}$. For the latter we are able to show that its solution is stable independently of $\eps$ and that, for $\eps\to 0^+$, it converges to the solution of the original problem under some extra regularity assumptions on the data and on the boundary. In the third section the Raviart-Thomas finite element is introduced together with our two discrete formulations, both depending on a mesh-dependent weighting parameter $\gamma$. As already mentioned, for the first one, $\gamma = h^{-1}$, while for the other $\gamma=h^{-\left(k+1\right)}$. In section 4 we prove the desired stability estimates, with respect to different mesh-dependent norms, guaranteeing the well-posedness of the associated problems. Then, in section 5, optimal a priori error estimates, in terms of the chosen norms, are demonstrated for the velocity and pressure fields. We demonstrate some super convergent results that enable us, for the two methods, to recover optimality for the $L^2$-error of the velocity field as well. Finally, some two-dimensional numerical examples  are provided in order to corroborate the theory.
\clearpage
\input{Sections/section1.tex}
\input{Sections/numerical_experiments.tex}
...
\clearpage
\bibliographystyle{plain}
\bibliography{bibliography}
\end{document}

%% file: Sections/section1.tex
\section{The Darcy problem and its variational formulation}
We introduce some useful notations for the forthcoming analysis. Let $D$ be a Lipschitz-regular domain (subset, open, bounded, connected) of $\mathbb R^d$, $d\in\{2,3\}$. Standard Sobolev spaces $H^s(D)$ for any $s\in\R$ and $H^t(\gamma)$ for $t\in\left[-1,1\right]$ are defined on the domain $D$ and on a non-empty open subset of its boundary $\gamma\subset\partial D$, see~\cite{adams}, with the convention $H^0(D):=L^2(D)$, $H^0(\gamma):=L^2(\gamma)$. Moreover, we introduce the following usual notations:
\begin{align*}
	L^2_0(D):=& L^2(D)/\R, \\
	H_{0,\gamma}(D):= & \text{ closure of } C_{0,\gamma}^\infty(D) \text{ with respect to } \norm{\cdot}_{H^s(D)}=\{v\in H(D): \restr{ v}{\gamma}=0 \}, \\
	H_{f,\gamma}(D):= &\{v\in H(D): \restr{ v}{\gamma}=f \}, \\
	\bm H^s(D):=&\left( H^s(D)\right)^d, \quad \bm H_{0,\gamma}^s(D):=\left( H_{0,\gamma}^s(D)\right)^d,\quad \bm H_{f,\gamma}^s(D):=\left( H_{f,\gamma}^s(D)\right)^d,\quad \bm H^s(\gamma):= \left( H^t(\gamma)\right)^d,\\
	\bm H(\dive;D):=& \{\vv\in \bm H^0(D): \dive \vv\in H^0(D)\},\\
	\bm H_{0,\gamma}(\dive;D):= & \text{ closure of } \left(C_{0,\gamma}^\infty(D)\right)^d \text{ with respect to } \norm{\cdot}_{H(\dive;D)}=\{\vv\in \bm H(\dive;D): \restr{\vv\cdot\n}{\gamma}=0 \}, \\
	\bm H_{\sigma,\gamma}(\dive;D):= & \{\vv\in \bm H(\dive;D): \restr{\vv\cdot\n}{\gamma}=\sigma \}, \\
	\bm H^{0}(\dive;D):= &\{\vv\in \bm H(\dive;D): \dive \vv=0 \},
\end{align*}
where the divergence operator and the traces on $\gamma$ are defined in the sense of distributions, see~\cite{monk}. For the sake of convenience we are going to employ the same notation $\abs{\cdot}$ for the volume (Lebesgue) and surface (Hausdorff) measures of $\R^d$.

We also denote as $\mathbb Q_{r,s,t}$ the vector space of polynomials of degree at most $r$ in the first variable, at most $s$ in the second and at most $t$ in the third one (analogously for th case $d=2$), $\mathbb P_u$ the vector space of polynomials of degree at most $u$. For the sake of simplicity of the notation we may write $\mathbb Q_k$ instead of $\mathbb Q_{k,k}$ or $\mathbb Q_{k,k,k}$.

Note that throughout this document $C$ will denote generic constants that may change at each occurrence, but that are always independent of the local mesh size.

 Let $\Omega$ be a Lipschitz-regular domain of $\R^d$, $d\in\{2,3\}$. We assume its boundary $\Gamma$ to be partitioned into $\Gamma=\Gamma_N\cup\Gamma_D$ with $\Gamma_N\cap\Gamma_D=\emptyset$.
 Let us consider the following problem, often associated to a linearized model for the flow of groundwater through our domain $\Omega$, here representing a saturated porus medium with permeability $\kappa$.
  Given $\f\in L^2(\Omega;\R^d)$, $g\in L^2(\Omega)$, $u_N\in H^{-\frac{1}{2}}(\Gamma_N)$, $p_D\in H^{\frac{1}{2}}(\Gamma_D)$, we look for $\left(\u,p\right)\in  H_{u_N,\Gamma_N}(\dive;\Omega)\times L^2(\Omega)$ such that
\begin{equation}\label{eq:prob_cont}
	\begin{cases}
		\kappa^{-1} \u -\nabla p = \f \qquad & \text{in}\;\Omega,\\
		\dive\u = g\qquad&\text{in}\;\Omega,\\
		\u\cdot\n=u_N\qquad&\text{on}\;\Gamma_N,\\
		p=p_D\qquad&\text{on}\;\Gamma_D.
	\end{cases}
\end{equation}
The unknowns $\u$ and $p$ represent, respectively, the seepage velocity and the pressure of the fluid.  The first equation of \eqref{eq:prob_cont} is called \emph{Darcy law} relating the velocity and the pressure gradient of the fluid, the second one expresses \emph{mass conservation}, the third and the fourth equations are, respectively, a \emph{Neumann boundary condition} for the velocity field and a \emph{Dirichlet boundary condition} for the pressure. Moreover, $\kappa\in\mathbb R^{d\times d}$ is symmetric positive definite with eigenvalues $\lambda_i$ such that $0<\lambda_{\min}\le\lambda_i\le\lambda_{\max}<+\infty$, for every $i=1,\dots,d$.

\begin{remark}
Contrary to the case of the Poisson problem, here Dirichlet boundary conditions for the pressure are \emph{natural}, in the sense that they can be implicitly enforced in the weak formulation of the problem, while Neumann boundary conditions for the velocity are \emph{essential}, i.e., they are imposed on the functional space. Moreover, let us observe that in the case of purely Neumann boundary conditions, in order to have well-posedness, we have to ``filter out'' the constant pressures, i.e., the trial and test functions for the pressures are required to lie in $ L^2_0(\Omega)$, and to impose a compatibility condition on the data: $\int_\Gamma u _N = \int_\Omega g$.
\end{remark}

\section{A perturbed formulation}
Find $\left(\u^\eps,p^\eps\right)\in  H(\dive;\Omega)\times L^2(\Omega)$ such that
\begin{equation}\label{eq:prob_cont_bis}
	\begin{cases}
		\kappa^{-1} \u^\eps -\nabla p^\eps = \f \qquad & \text{in}\;\Omega,\\
		\dive\u^\eps = g\qquad&\text{in}\;\Omega,\\
		\eps^{-1}\u^\eps\cdot\n=\eps^{-1} u_N -p^\eps\qquad&\text{on}\;\Gamma_N,\\
		p^\eps=p_D\qquad&\text{on}\;\Gamma_D.
	\end{cases}
\end{equation}
Note that as $\eps\to 0^+$ problem~\eqref{eq:prob_cont_bis} formally degenerates to~\eqref{eq:prob_cont}. In this sense~\eqref{eq:prob_cont_bis} is a perturbation of problem~\eqref{eq:prob_cont}. 

In the subsequent analysis we are going to consider, for the sake of simplicity, $\kappa=I$ the identity matrix. 

\begin{proposition}\label{prop:conv_eps}
	Let $(\u^\eps,p^\eps),(\u,p)$ be respectively the solutions to~\eqref{eq:prob_cont_bis} and~\eqref{eq:prob_cont}, then there exists $C>0$ such that
	\begin{align*}
\norm{\u-\u^\eps}_{L^2(\Omega)} \le  C \eps\left(\norm{\dive \f}_{L^2(\Omega)} + \norm{g}_{L^2(\Omega)} + \norm{u_N}_{H^{-\frac{1}{2}}(\Gamma_N)} + \norm{p_D}_{H^\frac{1}{2}(\Gamma_D)}	\right),
	\end{align*} 
\end{proposition}	
provided that $\f\in \bm H(\dive;\Omega)$.
\begin{proof}
Let us observe that if $\left( \u,p\right)$ and $\left( \u^\eps,p^\eps\right)$ solve, respectively, the problems~\eqref{eq:prob_cont} and~\eqref{eq:prob_cont_bis}, then $p$ and $p^\eps$ are the solutions of
 \begin{align}
&\begin{cases}\label{burman:eq10}
-\Delta p = -g +\dive \f \qquad&\text{in}\;\Omega,\\
\frac{\partial p}{\partial n} = u_N\qquad&\text{on}\;\Gamma_N,\\
p = p_D	\qquad&\text{on}\;\Gamma_D,
\end{cases}	
\\
&\begin{cases}\label{burman:eq11}
	-\Delta p^\eps = -g +\dive \f \qquad&\text{in}\;\Omega,\\
	\frac{\partial p^\eps}{\partial n} + \eps p^\eps = u_N \qquad&\text{on}\;\Gamma_N,\\
	p^\eps = p_D	\qquad&\text{on}\;\Gamma_D.
\end{cases}	
\end{align}
Let $\delta:= p -p^\eps$, with $p$ and $p^\eps$ respectively the solutions of~\eqref{burman:eq10} and~\eqref{burman:eq11}, then $\delta$ solves
\begin{align}&\begin{cases}\label{burman:eq12}
-\Delta \delta = 0 \qquad& \text{in}\;\Omega, \\
\delta +\eps^{-1}\frac{\partial \delta}{\partial n} = p \qquad &\text{on}\;\Gamma_N,\\
\delta = 0	\qquad&\text{on}\;\Gamma_D.
\end{cases}
\end{align}
We rewrite~\eqref{burman:eq12} in variational form. Find $\delta \in H^1_{0,\Gamma_D}(\Omega)$ such that
\begin{align}\label{burman:eq13}
\left( \nabla \delta,\nabla \varphi\right)_\Omega + \eps \langle \delta, \varphi\rangle_{\Gamma_N} = \eps \langle p,\varphi \rangle_{\Gamma_N}\qquad\forall\ \varphi\in H^1_{0,\Gamma_D}(\Omega).
\end{align}
From standard theory,~\eqref{burman:eq13} is well-posed and, in particular, the bilinear form inducing its left hand side is coercive, meaning that
\begin{align}\label{burman:eq14}
 \norm{\delta}^2_{H^1(\Omega)} \lesssim \norm{\nabla \delta}^2_{L^2(\Omega)} + \eps \norm{\delta}^2_{L^2(\Gamma_N)}.
\end{align}
By combining~\eqref{burman:eq13},~\eqref{burman:eq14}, the Cauchy-Schwarz and a standard trace inequality, we get
\begin{align}\label{burman:eq15}
	\norm{\delta}^2_{H^1(\Omega)} \lesssim \eps \norm{p}_{H^1(\Omega)}\norm{\delta}_{H^1(\Omega)}.
\end{align}
On the other hand, $p$ solves~\eqref{burman:eq10}, hence
\begin{align*}
\norm{\delta}^2_{H^1(\Omega)} \lesssim \eps \left(\norm{\dive \f}_{L^2(\Omega)} + \norm{g}_{L^2(\Omega)} + \norm{u_N}_{H^{-\frac{1}{2}}(\Gamma_N)} + \norm{p_D}_{H^{\frac{1}{2}}(\Gamma_D)}	\right)\norm{\delta}_{H^1(\Omega)}. 
\end{align*}	
 Since $\u-\u^\eps = \nabla p + \f - \left( \nabla p^\eps +\f \right) = \nabla \left( p-p^\eps \right) = \nabla \delta$, then we are done.
\end{proof}
In order to avoid technicalities, let us assume $\Omega$ to be a convex domain with  a $C^2$ boundary and the Neumann data to be homogeneous.
\begin{proposition}\label{prop:stability_eps_simple}
	Let $\left(\u^\eps,p^\eps\right)$ be the solution of~\eqref{eq:prob_cont_bis} and suppose that $\Omega$ is convex, $\Gamma$ is $C^2$ and $u_N=0$. Then there exists $C>0$, independent of $\eps$, such that
	\begin{align}\label{burman:eq20bis_simple}
	\norm{\u^\eps}_{H^{1}(\Omega)} \le  C \left( \norm{\f}_{H^1(\Omega)} + \norm{g}_{L^{2}(\Omega)}   \right),
\end{align}
provide that $\f\in \bm H^{1}(\Omega)$.
\end{proposition}
\begin{proof}
	As before, let us consider an equivalent formulation for~\eqref{eq:prob_cont_bis} and, without loss of generality, put ourselves in the pure Neumann case $\Gamma=\Gamma_N$. We observe that $\left(\u^\eps,p^\eps \right)$ is the solution to~\eqref{eq:prob_cont_bis} if and only if $p^\eps$ solves
	\begin{align}
		\begin{cases}\label{burman:eq16_simple}
			-\Delta p^\eps = -g +\dive \f \qquad&\text{in}\;\Omega,\\
			\frac{\partial p^\eps}{\partial n} + \eps p^\eps = 0 \qquad&\text{on}\;\Gamma.
		\end{cases}	
	\end{align}
	Let us recall that given $\w:\Omega\to \R^d$ and $\varphi:\Omega\to\R$, the following decompositions hold in an open neighborhood of the boundary
	\begin{align*}
		\w = \w_T + w_n \n,\qquad \nabla \varphi = \nabla_T \varphi + \frac{\partial\varphi}{\partial n}\n,
	\end{align*}
	$\w_T$ and $w_n$ being, respectively, the tangent and normal components of $\w$. The operator $\nabla_T:H^1(\Gamma)\to\bm L^2_T(\Gamma)$ is the tangential gradient and can be defined as in~\cite{monk}.
	Theorem~3.1.1.1 in~\cite{grisvard} states that if $\Omega$ is open, bounded with $\Gamma$ of class $C^2$, then every $\vv\in\bm H^1(\Omega)$ satisfies
	\begin{align}\label{burman:eq17_simple}
		\norm{\dive \vv}^2_{L^2(\Omega)} -\sum_{i,j=1}^d \left(D_j v_i, D_i v_j  \right)_{\Omega} = -2\langle \vv_T, \nabla_T v_n \rangle_\Gamma - \int_{\Gamma} \{ (\operatorname{tr}\mathcal B)v_n^2 +\mathcal B(\vv_T,\vv_T)  \}, 
	\end{align}
	$\mathcal B(\cdot,\cdot)$ being the second fundamental quadratic form associated to $\Gamma$ and $\operatorname{tr}\mathcal B$ its trace, see~\cite{grisvard} for the definitions. Let us apply~\eqref{burman:eq17_simple} to $\vv=\nabla p^\eps$. We have
	\begin{align}\label{burman:eq18_simple}
	\norm{\Delta p^\eps}^2_{L^2(\Omega)}- \sum_{i,j=1}^d \int_{\Omega} \abs{D^2_{x_ix_j}p^\eps}^2\ge 	-2\langle \nabla_T p^\eps, \nabla_T \frac{\partial p^\eps}{\partial n}\rangle_\Gamma,
	\end{align}	
	since
	\begin{align*}
		-\int_{\Gamma}	\mathcal B(\nabla_T p^\eps,\nabla_T p^\eps)\ge  0,\qquad -\int_{\Gamma} (\operatorname{tr}\mathcal B) \frac{\partial p^\eps}{\partial n}^2 \ge 0,
	\end{align*}	
	having used the non-positiveness of $\mathcal B$ due to the convexity of $\Omega$ (see~\cite{grisvard}). Moreover, the boundary conditions of~\eqref{eq:prob_cont_bis} imply
	\begin{align*}
		-2\langle \nabla_T p^\eps, \nabla_T \frac{\partial p^\eps}{\partial n}\rangle_\Gamma	=  2 \eps \norm{\nabla_Tp^\eps}^2_{L^2(\Gamma)},
	\end{align*}
	 hence~\eqref{burman:eq18_simple} in turn implies
	\begin{align*}
		\sum_{i,j=1}^d \int_{\Omega} \abs{D^2_{x_ix_j}p^\eps}^2 \le \norm{\Delta p^\eps}^2_{L^2(\Omega)} -2  \eps \norm{\nabla_T p^\eps}^2_{L^2(\Gamma)},
	\end{align*}
	and, in particular,
	\begin{align*}
		\sum_{i,j=1}^d \int_{\Omega} \abs{D^2_{x_ix_j}p^\eps}^2 \le \norm{\Delta p^\eps}^2_{L^2(\Omega)}.
	\end{align*}
	Let us bound the other terms which are part of $\norm{p^\eps}_{H^2(\Omega)}$. Green formula states
	\begin{align*}
		-\left(\Delta p^\eps, p^\eps   \right)_\Omega =\norm{\nabla p^\eps}^2_{L^2(\Omega)} - \langle \frac{\partial p^\eps}{\partial n} p^\eps\rangle_\Gamma,	
	\end{align*}	
	so that, using the boundary conditions of~\eqref{eq:prob_cont_bis},
	\begin{align*}
		\norm{\nabla p^\eps}^2_{L^2(\Omega)} =& 	-\left(\Delta p^\eps, p^\eps   \right)_\Omega +  \langle \frac{\partial p^\eps}{\partial n}, p^\eps\rangle_\Gamma =
		-\left(\Delta p^\eps, p^\eps   \right)_\Omega  - \langle \eps p^\eps, p^\eps \rangle_\Gamma \\
		\le & \norm{-\Delta p^\eps}_{L^2(\Omega)} \norm{p^\eps}_{L^2(\Omega)} -\eps \norm{p^\eps}^2_{L^2(\Gamma)},
	\end{align*}	
	implying
	\begin{align}\label{burman:eq20_simple}
		\norm{\nabla p^\eps}^2_{L^2(\Omega)} 
		\le  \norm{\Delta p^\eps}_{L^2(\Omega)} \norm{p^\eps}_{L^2(\Omega)}.
	\end{align}
	Because of the compatibility condition on the data when $\Gamma=\Gamma_N$, we have
	\begin{align*}
		\int_\Gamma p^\eps = \eps^{-1} \left( \int_\Gamma u_N - \int_\Gamma \frac{\partial p^\eps}{\partial n} \right) = \eps^{-1} \left( \int_\Gamma u_N - \int_\Omega  \dive \u^\eps \right) = \eps^{-1} \left( \int_\Gamma u_N - \int_\Omega g \right)=0,
	\end{align*}
	hence, using the Friedrichs inequality,
	\begin{align}\label{burman:eq19_simple}
		\norm{p^\eps}_{L^2(\Omega)}	\le C \left( \norm{\nabla p^\eps}_{L^2(\Omega)} + \int_\Gamma p^\eps \right) = 
		C \norm{\nabla p^\eps}_{L^2(\Omega)}.
	\end{align}	
	Combining~\eqref{burman:eq20_simple} and~\eqref{burman:eq19_simple}, we obtain
	\begin{align*}
		\norm{\nabla p^\eps}^2_{L^2(\Omega)} \le C \norm{\Delta p^\eps}_{L^2(\Omega)} \norm{\nabla p^\eps}_{L^2(\Omega)},
	\end{align*}	
	so that
	\begin{align*}
		\norm{\nabla p^\eps}_{L^2(\Omega)} \le C  \norm{\Delta p^\eps}_{L^2(\Omega)}.
	\end{align*}
	Hence,
	\begin{equation}\label{burman:eq21_simple}
		\begin{aligned}
		\norm{p^\eps}_{H^2(\Omega)} \le &\norm{p^\eps}_{L^2(\Omega)} + \norm{\nabla p^\eps}_{L^2(\Omega)} + \left(\sum_{i,j=1}^d \int_\Omega \abs{D^2_{x_i,x_j} p^\eps}^2 \right)^{\frac{1}{2}}	\le C \norm{\nabla p^\eps}_{L^2(\Omega)} + \norm{\Delta p^\eps}_{L^2(\Omega)}\\
		 \le &C \norm{\Delta p^\eps}_{L^2(\Omega)} \le C \left( \norm{\dive \f}_{L^2(\Omega)} + \norm{g}_{L^2(\Omega)}\right).
		 \end{aligned}
	\end{equation}	
Finally, using the relation $u^\eps = \nabla p^\eps +\f$, it holds
\begin{align*}
	\norm{\u^\eps}_{H^{1}(\Omega)} \le &\norm{\nabla p^\eps}_{H^{1}(\Omega)} + \norm{\f}_{H^{1}(\Omega)}\le \norm{p^\eps}_{H^{2}(\Omega)} + \norm{\f}_{H^{1}(\Omega)} \\
	\le & C  \left(  \norm{\dive \f}_{L^2(\Omega)} + \norm{g}_{L^2(\Omega)} \right) + \norm{\f}_{H^{1}(\Omega)}\\
	\le & C \left( \norm{\f}_{H^{1}(\Omega)} + \norm{g}_{L^2(\Omega)}   \right).
\end{align*}
\end{proof}
\begin{remark}\label{remark:bd}
Let us point out that the statement of Proposition~\ref{prop:stability_eps_simple} holds true when $\Omega$ is convex with a Lipschitz polygonal boundary $\Gamma$. We refer the interested reader to Remark~3.2.4.6 of~\cite{grisvard}.
\end{remark}
\section{The finite element discretization}\label{sec:fe_disc}
Let $\left(\mathcal T_h\right)_{h>0}$ denote a family of triangular or quadrilateral meshes  of $\Omega$. It will be useful to partition the collection of edges (or faces if $d=3$) $\mathcal F_h$ of $\mathcal T_h$ into three collections: the internal ones $\mathcal F_h^i$ and the ones lying on $\Gamma_N$ and on $\Gamma_D$, grouped respectively in $\mathcal F_h^\partial(\Gamma_N)$ and $\mathcal F_h^\partial (\Gamma_D)$, $\mathcal F_h^\partial = \mathcal F_h^\partial(\Gamma_N) \cup \mathcal F_h^\partial (\Gamma_D)$. For every $K\in\mathcal T_h$, $h>0$, let $h_K:=\operatorname{diam}(K)$ and $h:=\max_{K\in\mathcal T_h} h_K$. We assume the mesh to be \emph{shape-regular}, i.e., there exists $\sigma>0$, independent of $h$, such that $\max_{K\in\mathcal T_h} \frac{h_K}{\rho_K}\le \sigma$, $\rho_K$ being the diameter of the largest ball inscribed in $K$. Moreover, $\mathcal T_h$ is supposed to be \emph{quasi-uniform} in the sense that there exists $\tau>0$, independent of $h$, such that $\min_{K\in \mathcal T_h} h_K\ge \tau h$.
Let $\varphi:\Omega\to\R$ be smooth enough so that for every $K\in\mathcal T_h$ its restriction $\restr{\varphi}{K}$ can be extended up to the boundary $\partial K$. Then, for all $f\in\mathcal F_h^i$ and a.e. $x\in f$, we define the \emph{jump} of $\varphi$ as
\begin{align*}
	[\varphi]_f(x):=\restr{\varphi}{K_1}(x)-\restr{\varphi}{K_2}(x),
\end{align*}
where $f=\partial K_1\cap \partial K_2$. We may remove the subscript $f$ when it is clear from the context to which \emph{facet} (edge if $d=2$, face if $d=3$) we refer to.

In order to discretize problem~\eqref{eq:prob_cont}, we need to choose a suitable couple of subspaces $ V_h\subset H\left(\dive;\Omega\right)$ and $ Q_h\subset L^2(\Omega)$. In the following, $\hat K$ will be our \emph{reference element}, and, according to the type of mesh employed, it will be either the unit $d$-simplex, i.e., the triangle of vertices $(0,0),(1,0),(0,1)$, or the unit $d$-cube $\left[0,1\right]^d$. 
 
 For the triangular meshes, the \emph{Raviart-Thomas} finite element on $\hat K$ is
 \begin{equation*}
 	\mathbb{RT}_k(\hat K):=\left(\mathbb{P}_k(\hat K)\right)^d\oplus \bm x\tilde{\mathbb P}_k(\hat K),	
\end{equation*}
 while, in the case of quadrilaterals, it reads as follows (see, for instance, \cite{arnold}):
\begin{equation*}
	\mathbb{RT}_k(\hat K):=
	\begin{cases}
		\mathbb{Q}_{k+1,{k}}(\hat K)\times\mathbb{Q}_{k,k+1}(\hat K)\qquad&\text{if}\;d=2,\\
		\mathbb{Q}_{k+1,{k},k}(\hat K)\times \mathbb{Q}_{k,k+1,k}(\hat K)\times\mathbb{Q}_{k,k,k+1}(\hat K)\qquad&\text{if}\;d=3.
	\end{cases}
\end{equation*}
We map the reference element  to a general $K\in\mathcal T_h$ via the \emph{affine map} $F_K:\hat K\to K$, $F_K(\hat x):= B_K \hat x + b_K$, where $B_K\in \mathbb R^{d\times d}$ is diagonal and invertible, and $b_K\in\mathbb R^d$. For $H(\dive;\Omega)$, the natural way to transform functions from $\hat K$ to $K$ is through the \emph{Piola transform}. Namely, given $\hat \u:\hat K\to\mathbb R^d$, we define $\u=\mathcal P_K \hat \u: K\to\mathbb R^d$ by
\begin{equation*}
	\u(x)=\mathcal P_{K} \hat \u (x):=\abs{\det\left(B_K \right)}^{-1}B_K \hat\u(\hat x),\qquad\text{where}\qquad \hat x=B^{-1}_K\left(x-b_K\right).
\end{equation*}
For functions in $L^2(\Omega)$, we just compose them with the affine map, namely $\hat q:\hat K\to\R$ is transformed to $q=\hat q \circ F_K^{-1}:K\to\R$. In this way, we can define the finite-dimensional subspaces:
\begin{equation*}
\begin{aligned}
	V_h&:=\{\vv_h \in H\left(\dive;\Omega\right):\restr{\vv_h}{K}\in \mathbb{RT}_k(K)\quad\forall\ K\in\mathcal T_h\},\\
	Q_h&:=\{q_h\in L^2\left(\Omega\right):\restr{q_h}{K}\circ F_K\in\mathbb{P}_{k}\left(K\right)\quad\forall\ K\in\mathcal T_h\},\qquad\text{for triangles,}\\
	Q_h&:=\{q_h\in L^2\left(\Omega\right):\restr{q_h}{K}\circ F_K\in\mathbb{Q}_{k}\left(K\right)\quad\forall\ K\in\mathcal T_h\},\qquad\text{for quadrilaterals,}
\end{aligned}
\end{equation*}
where $\mathbb{RT}_k(K) :=\{\mathcal P_{K}\hat\w_h: \hat\w_h\in\mathbb {RT}_k(\hat K) \}$.  Remember that in the pure Neumann case, i.e., $\Gamma=\Gamma_N$, we have to filter out constant discrete pressures by imposing the zero average constraint to the space $Q_h$. 
Let us construct the interpolation operator onto the discrete velocities $r_h:\prod_{K\in\mathcal T_h}\bm H^s(K^{\mathrm{o}})  \to V_h$, by gluing together the local interpolation operators $r_K:\bm H^s(K^{\mathrm{o}})\to \mathbb {RT}_k(K)$, $s>\frac{1}{2}$, $K^{\mathrm{o}}$ denoting the interior of $K$, and using the natural degrees of freedom of the Raviart-Thomas finite element. For every $\vv\in \bm H^s(K^{\mathrm{o}})$, $s>\frac{1}{2}$, $r_K$ is uniquely defined by:
\begin{equation*}
	\begin{cases}
	\langle r_K \vv\cdot \n_e, q_h \rangle_f= \langle \vv\cdot\n_e, q_h\rangle_f \qquad&\forall\ q_h\in\Psi_k(f), \\
		\left(r_K\vv, \w_h\right)_K = \left(\vv, \w_h\right)_K \qquad&\forall\ \w_h\in \Psi_k\left(K\right)\qquad\text{if}\;k>0,\\
	\end{cases}
\end{equation*}
where, for triangles,
\begin{equation*}
	\Psi_k\left(K\right):=\left(\mathbb P_{k-1}(K)\right)^d,\qquad \Psi_k(f):=\mathbb P_{k}(f),\\
\end{equation*}
and, for quadrilaterals,
\begin{equation*}
	\Psi_k\left(K\right):=
	\begin{cases}
		\mathbb{Q}_{k-1,{k}}(K)\times\mathbb{Q}_{k,k-1}(K)\qquad&\text{if}\;d=2,\\
		\mathbb{Q}_{k-1,{k},k}(K)\times \mathbb{Q}_{k,k-1,k}(K)\times\mathbb{Q}_{k,k,k-1}(K)\qquad&\text{if}\;d=3,
	\end{cases}
\qquad
	\Psi_k\left(f\right):=
\begin{cases}
\mathbb P_k(f)\qquad&\text{if}\;d=2,\\
\mathbb Q_k(f)\qquad&\text{if}\;d=3,
\end{cases}
\end{equation*}
for all facets $f$ and triangles or quadrilaterals $K$ of $\mathcal T_h$. The natural choice in order to interpolate onto $Q_h$ is to employ an elementwise $L^2$-orthogonal projection, i.e., $\Pi_h: L^2\left(\Omega\right)\to Q_h$ such that for every $K\in\mathcal T_h$, $\restr{\Pi_h}{K}:=\Pi_K$, where for $\xi\in L^2(\Omega)$,
\begin{align*}
\left(\Pi_K \xi, q_h\right)_K &= \left(\xi, q_h\right)_K \qquad \forall\ q_h\in Q_h.
\end{align*}
for every $K\in\mathcal T_h$.
\begin{remark}
It is worth mentioning that the following numerical analysis remains valid if we employ another H(div)-conforming discretization instead, such as the so-called Brezzi-Douglas-Marini mixed element~\cite{brezzi_boffi}.
\end{remark}
\begin{proposition}\label{prop:cd}
	The following diagram commutes:
	\begin{equation}\label{comm_diag}
		\begin{CD}
			H(\dive;\Omega) \cap \prod_{K\in\mathcal T_h}\bm H^s(K^{\mathrm{o}})   @>\dive>> L^2(\Omega) \\
			@VVr_hV @VV\Pi_hV \\
			V_h @>\dive>> Q_h.
		\end{CD}	
	\end{equation}
In particular, it holds
\begin{align*}
\dive V_h= Q_h.
\end{align*}
\end{proposition}	
\begin{proof}
For the commutative diagram note that, for every $\vv\in H(\dive;\Omega) \cap \prod_{K\in\mathcal T_h}\bm H^s(K^{\mathrm{o}}) $, it holds
\begin{align*}
\left( \Pi_h \dive \vv,\varphi_h \right)_K	=& \left(\dive \vv,\varphi_h \right)_K = -\left( \vv,\nabla \varphi_h\right)_K +\langle \varphi_h, \vv\cdot\n\rangle_{\partial K}\\
 = &-\left( r_h\vv,\nabla \varphi_h\right)_K +\langle \varphi_h, r_h\vv \cdot\n\rangle_{\partial K} = \left(\dive r_h\vv ,\varphi_h\right)_K\qquad\forall\ \varphi_h\in\mathbb P_k(K) \;(\text{resp.}\; \mathbb Q_k(K)).
\end{align*}	
A direct calculation readily shows the inclusion $\dive V_h \subseteq Q_h$. Let us prove the other one. Let $q_h\in Q_h$, then, by the surjectivity of $\dive: \bm H^1(\Omega)\to\ L^2(\Omega)$~\cite{brezzi_boffi}, there exists $\vv\in H(\dive;\Omega)\cap \prod_{K\in\mathcal T_h}\bm H^s(K^{\mathrm{o}})\subseteq \bm H^1(\Omega)$ such that $\dive\vv=q_h$. Let us define $\vv_h:= r_h\vv$. Thanks to the commutativity diagram we have $\dive \vv_h=q_h$.
\end{proof}

We are now ready to introduce the discrete formulations we want to analyze. 

\subsubsection{First formulation}
Find $\left(\u_h,p_h\right)\in V_h\times Q_h$ such that
\begin{align}\label{eq:prob_disc}
\begin{cases}
a_h(\u_h,\vv_h) + b_1(\vv_h,p_h) =\left(\f,\vv_h\right)_{\Omega}+   h^{-1}\langle u_N,\vv_h\cdot\n \rangle_{\Gamma_N}+\langle p_D,\vv_h\cdot\n\rangle_{\Gamma_D}\qquad&\forall\ \vv_h\in V_h,\\
b_m(\u_h,q_h) =\left(g,q_h\right)_{\Omega} - m \langle q_h, u_N \rangle_{\Gamma_N}\qquad\qquad\qquad\qquad&\forall\ q_h\in Q_h,
\end{cases}
\end{align}
where $m\in \{0,1\}$. Here,
\begin{align}
	a_h(\u_h,\vv_h):=\left(\u_h,\vv_h\right)_{\Omega}+  h^{-1} \langle\u_h\cdot\n,\vv_h\cdot\n \rangle_{\Gamma_N}\qquad&\forall\ \u_h,\vv_h\in V_h,\\
	b_m(\u_h,p_h):= \left(p_h,\dive \u_h\right)_{\Omega}- m \langle p_h,\u_h\cdot\n \rangle_{\Gamma_N}\qquad&\forall\ \u_h\in V_h,p_h\in Q_h.
\end{align}
In what follows just the analysis for the symmetric case $m=1$ will be presented, however numerical results will be provided for the case $m = 0$ as well.

\subsubsection{Second formulation}
Find $\left(\u_h,p_h\right)\in V_h\times Q_h$ such that
\begin{align}\label{eq:prob_disc_bis}
\begin{cases}
a_\eps(\u_h,\vv_h) + b_0(\vv_h,p_h) =\left(\f,\vv_h\right)_{\Omega} + \langle\eps^{-1} u_N,\vv_h\cdot\n \rangle_{\Gamma_N}+\langle p_D,\vv_h\cdot\n\rangle_{\Gamma_D}\qquad&\forall\ \vv_h\in V_h,\\
b_0(\u_h,q_h) =\left(g,q_h\right)_{\Omega} \qquad\qquad\qquad\qquad&\forall\ q_h\in Q_h.
\end{cases}
\end{align}
where
\begin{align*}
a_\eps(\u_h,\vv_h):=\left(\u_h,\vv_h\right)_{\Omega}+ \eps^{-1} \langle\u_h\cdot\n,\vv_h\cdot\n \rangle_{\Gamma_N}\qquad&\forall\ \u_h,\vv_h\in V_h.
\end{align*}
\begin{remark}
	Let us observe that the non-symmetric version of problem~\eqref{eq:prob_disc}, i.e., with $m=0$, and formulation~\eqref{eq:prob_disc_bis}, thanks to Proposition~\ref{prop:cd} allows for a \emph{weakly divergence-free} numerical solution $\u_h$, namely $\dive \u_h=0$ in the sense of $L^2$, provided that the right hand side $g$ vanishes.
\end{remark}
\begin{lemma}
Formulations~\eqref{eq:prob_disc} and~\eqref{eq:prob_disc_bis} are consistent discretizations of~\eqref{eq:prob_cont} and~\eqref{eq:prob_cont_bis} respectively.
\end{lemma}
\begin{proof}
	It is clear that~\eqref{eq:prob_disc} is a consistent discetization of~\eqref{eq:prob_cont}.
	Let $\left(\u^\eps,p^\eps\right)$ be the solution to~\eqref{eq:prob_cont_bis}. Of course, we have
	\begin{align*}
	b_0(\u^\eps,q_h)=\left(g,q_h\right)\qquad\forall\ q_h\in Q_h.	
	\end{align*}		
	By integrating by parts the first equation of~\eqref{eq:prob_cont_bis}, we obtain
	\begin{align}\label{perturb5}
	\left(\u^\eps,\vv_h\right)_{\Omega} + b_0(\vv_h,p^\eps)-\langle p^\eps, \vv_h\cdot\n\rangle_{\Gamma_N} =\left(\f,\vv_h\right)_{\Omega}+\langle p_D,\vv_h\cdot\n \rangle_{\Gamma_D}\qquad\forall\ \vv_h\in V_h.
	\end{align}
	By performing \emph{static condensation} of the multiplier from the boundary conditions, we obtain
	\begin{align}\label{perturb6}
	p^\eps	= \eps^{-1}\left( \u_N-\u^\eps\cdot\n  \right)\qquad\text{on}\;\Gamma_N.
	\end{align}	
	Substituting~\eqref{perturb6} back into~\eqref{perturb5}, we obtain
	\begin{align*}
	a_\eps(\u^\eps,\vv_h) + b_0(\vv_h,p^\eps) =\left(\f,\vv_h\right)_{\Omega}+ \langle \eps^{-1}\u_N ,\vv_h\cdot\n \rangle_{\Gamma_N}+\langle p_D,\vv_h\cdot\n\rangle_{\Gamma_D}\qquad\forall\ \vv_h\in V_h.
	\end{align*}
\end{proof}	
For the numerical analysis of~\eqref{eq:prob_disc}, we endow the discrete spaces with the following mesh-dependent norms
\begin{align*}
	\norm{\vv_h}^2_{0,h}&:= \norm{\vv_h}^2_{L^2(\Omega)} + \sum_{f\in\mathcal F_h^\partial(\Gamma_N)} h^{-1}\norm{\vv_h\cdot\n}^2_{L^2(f)},  \\
	\norm{q_h}^2_{1,h}&:= \sum_{K\in\mathcal T_h}\norm{\nabla q_h}^2_{L^2(K)} + \sum_{f\in\mathcal F_h^i}h^{-1}\norm{\left[q_h\right]}^2_{L^2(f)} + \sum_{f\in\mathcal F_h^\partial(\Gamma_D)}h^{-1}\norm{q_h}^2_{L^2(f)},
\end{align*}
while for~\eqref{eq:prob_disc_bis} we are going to employ:
\begin{align*}
	\norm{\vv_h}^2_{0,h,\eps}&:= \norm{\vv_h}^2_{L^2(\Omega)} + \sum_{f\in\mathcal F_h^\partial(\Gamma_N)} \eps^{-1} \norm{\vv_h\cdot\n}^2_{L^2(f)},  \\
		\norm{q_h}^2_{1,h,\eps}&:= \sum_{K\in\mathcal T_h}\norm{\nabla q_h}^2_{L^2(K)} + \sum_{f\in\mathcal F_h^i}h^{-1}\norm{\left[q_h\right]}^2_{L^2(f)} + \sum_{f\in\mathcal F_h^\partial}h^{-1}\norm{q_h}^2_{L^2(f)},
\end{align*}
for every $\vv_h\in V_h$ and $q_h\in Q_h$.
\begin{remark}
	Informally speaking, the idea of both approaches is to unbalance the norms in order to go back to the elliptic case. Note that the natural functional setting for the mixed formulation of the Poisson problem is $H(\dive;\Omega)\times L^2(\Omega)$, but here we consider norms that induce the same topology as that of $\left[L^2(\Omega)\right]^d\times H^1(\Omega)$. Moreover, we observe that in both formulations~\eqref{eq:prob_disc} and~\eqref{eq:prob_disc_bis} a superpenalty parameter $\gamma$ is imposed in the flux variable. Indeed, the natural weight, mimicking the $H^{-\frac{1}{2}}$-scalar product, would be $h \langle\u_h\cdot\n, \vv_h\cdot\n \rangle_{\Gamma_N}$. However such a weight does not lead to an optimally converging scheme. In addition, this is also what destroys the conditioning (see subsection~\ref{conditioning}).
%
\end{remark}
\section{Stability estimates}
In this section we carry on at the same time the proofs of the well-posedness of the two discrete formulations.
\begin{proposition}
	There exist $M_{a_h}, M_{a_{\eps}}, M_{b_m}>0$, $m=0,1$, such that	
	\begin{align*}	
		\abs{a_h (\w_h,\vv_h)}	\le M_{a_h} \norm{\w_h}_{0,h}\norm{\vv_h}_{0,h}\qquad &\forall\ \w_h,\vv_h\in V_h,\\
		\abs{a_{\eps} (\w_h,\vv_h)}	\le M_{a_\eps} \norm{\w_h}_{0,h,\eps}\norm{\vv_h}_{0,h,\eps}\qquad &\forall\ \w_h,\vv_h\in V_h,\\
		\abs{b_1(\vv_h,q_h)} \le  M_{b_1} \norm{\vv_h}_{0,h}\norm{q_h}_{1,h}\qquad &\forall\ \vv_h\in V_h,\ q_h\in Q_h, \\
		\abs{b_0(\vv_h,q_h)} \le  M_{b_0} \norm{\vv_h}_{0,h,\eps}\norm{q_h}_{1,h,\eps}\qquad &\forall\ \vv_h\in V_h,\ q_h\in Q_h.
	\end{align*}	
\end{proposition}	
\begin{proof}
	Let $\w_h,\vv_h\in V_h$, $q_h\in Q_h$ be arbitrary. It holds
	\begin{align*}
\abs{a_h(\w_h,\vv_h)}\le &\norm{\w_h}_{L^2(\Omega)} \norm{\vv_h}_{L^2(\Omega)} + h^{-\frac{1}{2}}\norm{\vv_h\cdot\n}_{L^2(\Gamma_N)}h^{-\frac{1}{2}}\norm{\w_h\cdot\n}_{L^2(\Gamma_N)}\le \norm{\w_h}_{0,h} \norm{\vv_h}_{0,h},\\
\abs{a_\eps(\w_h,\vv_h)}\le& \norm{\w_h}_{L^2(\Omega)} \norm{\vv_h}_{L^2(\Omega)} + \eps^{-\frac{1}{2}}\norm{\vv_h\cdot\n}_{L^2(\Gamma_N)}\eps^{-\frac{1}{2}}\norm{\w_h\cdot\n}_{L^2(\Gamma_N)}\le \norm{\w_h}_{0,h,\eps} \norm{\vv_h}_{0,h,\eps}.
	\end{align*}
By integration by parts, we get
	\begin{align*}
		b_1(\vv_h,q_h) =& \left(q_h,\dive\vv_h\right)_\Omega - \langle q_h,\vv_h\cdot\n\rangle_{\Gamma_N} = \sum_{K\in\mathcal T_h} \left(q_h,\dive \vv_h\right)_K - \sum_{f\in\mathcal F_h^\partial(\Gamma_N)}\langle q_h,\vv_h\cdot\n\rangle_f\\
		=&-\sum_{K\in\mathcal T_h} \left(\nabla q_h,\vv_h\right)_K + \sum_{f\in\mathcal F_h^i}\langle \left[q_h \right],\vv_h\cdot\n \rangle_f+ \sum_{f\in\mathcal F_h^\partial(\Gamma_D)}\langle q_h ,\vv_h\cdot\n \rangle_f.
	\end{align*}
	Thus,
	\begin{align*}
		\abs{b_1(\vv_h,q_h)}\le& \sum_{K\in\mathcal T_h}\norm{\nabla q_h}_{L^2(K)}\norm{\vv_h}_{L^2(K)} + \sum_{f\in\mathcal F_h^i} h^{-\frac{1}{2}}\norm{\left[q_h\right]}_{L^2(f)}h^{\frac{1}{2}}\norm{\vv_h\cdot\n}_{L^2(f)} \\
		 &+ \sum_{f\in\mathcal F_h^\partial(\Gamma_D)} h^{-\frac{1}{2}}\norm{q_h}_{L^2(f)}h^{\frac{1}{2}}\norm{\vv_h\cdot\n}_{L^2(f)}.
	\end{align*}
We recall some standard inverse inequalities, namely,
\begin{equation}
	\begin{aligned}
		h^{\frac{1}{2}} \norm{\vv_h\cdot\n}_{L^2(f)} \lesssim \norm{\vv_h}_{L^2(K)}\qquad f\in\mathcal F_h^\partial(\Gamma_D), f\in\mathcal F_h^i, f\subset \partial K \label{eq:disc_inv_ineq}.
\end{aligned}
\end{equation}
In this way we obtain
\begin{align*}
\abs{b_1(\vv_h,q_h)}\lesssim 	& \norm{\vv_h}_{0,h} \norm{q_h}_{1,h}.
\end{align*}
 On the other hand,
	\begin{align*}
	b_0(\vv_h,q_h) =& -\sum_{K\in\mathcal T_h} \left(\nabla q_h,\vv_h\right)_K +\langle q_h,\vv_h\cdot\n\rangle_{\partial K} = 
	-\sum_{K\in\mathcal T_h} \left(\nabla q_h,\vv_h\right)_K+ \sum_{f\in\mathcal F_h^i} \langle \left[q_h\right],\vv_h\cdot\n\rangle_f\\
	&  + \sum_{f\in\mathcal F_h^\partial(\Gamma_N)} \langle q_h,\vv_h\cdot\n\rangle_f   + \sum_{f\in\mathcal F_h^\partial(\Gamma_D)} \langle q_h,\vv_h\cdot\n\rangle_f.
	\end{align*}
We have
	\begin{align*}
	\abs{b_0(\vv_h,q_h)}\le & \sum_{K\in\mathcal T_h} \norm{\nabla q_h}_{L^2(K)}\norm{\vv_h}_{L^2(K)}+\sum_{f\in\mathcal F_h^i} h^{-\frac{1}{2}}\norm{\left[q_h\right]}_{L^2(f)}h^{\frac{1}{2}}\norm{\vv_h\cdot\n}_{L^2(f)} \\
	&+\sum_{f\in\mathcal F_h^\partial} h^{-\frac{1}{2}}\norm{q_h}_{L^2(f)}h^{\frac{1}{2}}\norm{\vv_h\cdot\n}_{L^2(f)} \lesssim \norm{\vv_h}_{0,h,\eps}\norm{q_h}_{1,h,\eps},
	\end{align*}
having used again~\eqref{eq:disc_inv_ineq} and $h^{-\frac{1}{2}}<\eps^{-\frac{1}{2}}$ for $\eps \ll h$.  
\end{proof}
\begin{proposition}
	There exist ${\alpha_{a_h}},\alpha_{a_\eps}>0$ such that
	\begin{align*}
		a_{h}(\vv_h,\vv_h)\ge &\alpha_{a_h} \norm{\vv_h}^2_{0,h}\qquad\forall\ \vv_h\in V_h,\\
		a_{\eps}(\vv_h,\vv_h)\ge &\alpha_{a_\eps} \norm{\vv_h}^2_{0,h,\eps}\qquad\forall\ \vv_h\in V_h.
	\end{align*}
	\begin{proof}
		Let us take $\vv_h\in V_h$ arbitrary and compute
		\begin{align*}
			a_h(\vv_h,\vv_h) = &\norm{\vv_h}^2_{L^2(\Omega)}+  h^{-1}\norm{\vv_h\cdot\n}^2_{L^2(\Gamma_N)} 
		= \norm{\vv_h}^2_{0,h}.
		\end{align*}
		 The other coercivity estimate follows anologously. Hence, $\alpha_{a_h}=\alpha_{a_\eps}=1$.
	\end{proof}
\end{proposition}
\begin{proposition}\label{prop:infsup_fitted}
	There exist $\beta_m>0$, $m\in\{0,1\}$, such that	
	\begin{align*}	
	\inf_{q_h\in Q_h}\sup_{\vv_h\in V_h} \frac{b_1(\vv_h,q_h)}{\norm{\vv_h}_{0,h}\norm{q_h}_{1,h}}	\ge &\beta_1,\\
	\inf_{q_h\in Q_h}\sup_{\vv_h\in V_h} \frac{b_0(\vv_h,q_h)}{\norm{\vv_h}_{0,h,\eps}\norm{q_h}_{1,h,\eps}}	\ge& \beta_0.
	\end{align*}	
\end{proposition}	
\begin{proof}
	We start with $m=1$. Let us fix $q_h\in Q_h$ arbitrary. We construct $\vv_h$ by using the dofs of the Raviart-Thomas space.
	\begin{align}
		\langle \vv_h\cdot\n,\varphi_h \rangle_f =  h^{-1} \langle \left[q_h\right],\varphi_h\rangle_f\qquad&\forall\ f\in\mathcal F_h^i,\ \varphi_h\in\Psi_k(f),\label{stab:eq1}\\
		\langle \vv_h\cdot\n,\varphi_h\rangle_f =0\qquad&\forall\ f\in\mathcal F_h^{\partial}(\Gamma_N),\ \varphi_h\in\Psi_k(f), \label{stab:eq2}\\
				\langle \vv_h\cdot\n,\varphi_h\rangle_f =h^{-1} \langle q_h,\varphi_h\rangle_f\qquad&\forall\ f\in\mathcal F_h^{\partial}(\Gamma_D),\ \varphi_h\in\Psi_k(f), \label{stab:eq4}\\
		\left(\vv_h,\bm\psi_h \right)_K = -\left(\nabla q_h,\bm\psi_h \right)_K\qquad&\forall\ K\in\mathcal T_h,\ \bm\psi_h\in\Psi_k(K),\ \text{if}\ k>0. \label{stab:eq3}
	\end{align}	
	By using the definition of $\vv_h$,
	\begin{align*}
		b_1(\vv_h,q_h) =&  \left(q_h,\dive\vv_h \right)_{\Omega} - \langle q_h,\vv_h\cdot\n\rangle_{\Gamma_N} = \sum_{K\in\mathcal T_h} \left(q_h,\dive\vv_h\right)_K - \sum_{f\in\mathcal F_h^{\partial}(\Gamma_N)} \langle q_h,\vv_h\cdot\n\rangle_f \\
		=&-\sum_{K\in\mathcal T_h}\left( \nabla q_h,\vv_h \right)_K  + \langle q_h, \vv_h\cdot\n\rangle_{\partial K} - \sum_{f\in\mathcal F_h^{\partial}(\Gamma_N)} \langle q_h,\vv_h\cdot\n\rangle_f \\
		=& -\sum_{K\in\mathcal T_h}\left( \nabla q_h,\vv_h \right)_K +\sum_{f\in\mathcal F_h^i} \langle \left[q_h\right],\vv_h\cdot\n\rangle_f +\sum_{f\in\mathcal F_h^{\partial}(\Gamma_D)} \langle q_h,\vv_h\cdot\n\rangle_f\\
		=& \sum_{K\in\mathcal T_h} \norm{\nabla q_h}^2_{L^2(K)} + \sum_{f\in\mathcal F_h^i}h^{-1} \norm{\left[q_h\right]}^2_{L^2(f)} + \sum_{f\in\mathcal F_h^\partial(\Gamma_D)}h^{-1}\norm{q_h}^2_{L^2(f)}  = \norm{q_h}^2_{1,h}.
	\end{align*}
		Finally, let us show that $\norm{\vv_h}_{0,h}\le C \norm{q_h}_{1,h}$. Note that for every $f\in\mathcal F_h^{\partial}(\Gamma_N)$, since $\restr{\vv_h\cdot\n}{f} \in\mathbb P_k(f)$, \eqref{stab:eq2} implies
	\begin{align*}
	\norm{\vv_h\cdot\n }^2_{L^2(f)} = \langle \vv_h\cdot\n,\vv_h\cdot\n \rangle_{f} =0 \qquad\Rightarrow\qquad \norm{\vv_h\cdot\n}_{L^2(f)}=0.
	\end{align*}
	Then, let us show $\norm{\vv_h}_{L^2(\Omega)} \le C \norm{ q_h}_{1,h}$. From~\eqref{stab:eq1} it holds $\restr{\vv_h\cdot\n}{f}=h_K^{-1} \restr{\pi_{f,k}\left[q_h\right] }{f}$ for every $f\in\mathcal{F}_h^i$ and from~\eqref{stab:eq3} we have $\restr{\pi_{K,k}\vv_h}{K} =- \restr{\pi_{K,k}\nabla q_h}{K}$ for every $K\in\mathcal T_h$. Note that here $\pi_{K,k}$ denotes the $L^2$-orthogonal projection onto $\Psi_k(K)$. Similarly, $\pi_{f,k}$ is the $L^2$-projection onto $\Psi_k(f)$. From finite dimensionality it holds $\norm{\hat \vv_h}^2_{L^2(\hat K)} \lesssim \norm{\pi_{\hat K,k}\hat \vv_h}^2_{L^2(\hat K)} + \norm{\hat \vv_h\cdot\hat \n}^2_{L^2(\hat f)}$.  Hence, $\norm{\vv_h}^2_{L^2(K)}\lesssim \norm{\nabla q_h}^2_{L^2(K)}+h_K^{-1}\norm{[q_h]}^2_{L^2(f)}$, $f$ being a facet of $K$, which follows by a standard scaling argument (see Proposition 2.1 of~\cite{doi:10.1137/17M1163335}) and  by construction of $\vv_h$.
	
	Let us now take $m=0$ and $q_h\in Q_h$. We define $\vv_h$ as follows:
		\begin{align}
	\langle \vv_h\cdot\n,\varphi_h \rangle_f = h^{-1} \langle \left[q_h\right],\varphi_h\rangle_f\qquad&\forall\ f\in\mathcal F_h^i,\ \varphi_h\in\Psi_k(f),\label{stab:eq1bis}\\
	\langle \vv_h\cdot\n,\varphi_h\rangle_f = h^{-1}\langle q_h,\varphi_h\rangle_f \qquad&\forall\ f\in\mathcal F_h^{\partial}(\Gamma_N),\ \varphi_h\in \Psi_k(f),\label{stab:eq2bis}\\
		\langle \vv_h\cdot\n,\varphi_h\rangle_f = h^{-1}\langle q_h,\varphi_h\rangle_f  \qquad&\forall\ f\in\mathcal F_h^{\partial}(\Gamma_D),\ \varphi_h\in\Psi_k(f),\label{stab:eq4bis}\\
	\left(\vv_h,\bm\psi_h \right)_K = -\left(\nabla q_h,\bm\psi_h \right)_K\qquad&\forall\ K\in\mathcal T_h,\ \bm\psi_h\in\Psi_k(K),\ \text{if}\ k>0\label{stab:eq3bis}. 
	\end{align}	
	\begin{align*}
		b_0(\vv_h,q_h) =&  \left(q_h,\dive\vv_h \right)_{\Omega} = \sum_{K\in\mathcal T_h} \left(q_h,\dive\vv_h\right)_K=-\sum_{K\in\mathcal T_h}\left( \nabla q_h,\vv_h \right)_K  +  \langle q_h, \vv_h\cdot\n\rangle_{\partial K}  \\
		=& -\sum_{K\in\mathcal T_h}\left( \nabla q_h,\vv_h \right)_K +\sum_{f\in\mathcal F_h^i} \langle \left[q_h\right],\vv_h\cdot\n\rangle_f +\sum_{f\in\mathcal F_h^\partial(\Gamma_N)} \langle q_h,\vv_h\cdot\n\rangle_f +\sum_{f\in\mathcal F_h^\partial(\Gamma_D)} \langle q_h,\vv_h\cdot\n\rangle_f\\
		=& \sum_{K\in\mathcal T_h} \norm{\nabla q_h}^2_{L^2(K)}+ \sum_{f\in\mathcal F_h^i} h^{-1}\norm{\left[q_h\right]}^2_{L^2(f)}+\sum_{f\in\mathcal F_h^\partial}h^{-1}\norm{q_h}^2_{L^2(f)}=\norm{q_h}^2_{1,h,\eps}.
	\end{align*}
	We refer to~\cite{konno_robin} for the inequality $\norm{\vv_h}_{0,h,\eps}\le C \norm{q_h}_{1,h,\eps}$. 
\end{proof}

\section{A priori error estimates}
In this section we will prove a priori error estimates for the formulations~\eqref{eq:prob_disc} and~\eqref{eq:prob_disc_bis}.
\\
We observe that all the constants appearing throughout this section and concerning the error bounds for the formulation~\eqref{eq:prob_disc_bis} are independent of the parameter $\eps$. This is due to the orthogonality properties of the interpolants along the boundary.
  \begin{lemma}\label{apriori:lemma1_m=1}
  	Let $\left(\bm\u,p\right)$ be the solution of the continuous problem~\eqref{eq:prob_cont} and $\left(\bm\u_h,p_h\right)\in V_h\times Q_h$ the one of the discrete problem~\eqref{eq:prob_disc} with $m=1$. Then
  	\begin{equation}
  	\norm{\bm\u_h - r_h\bm\u}_{0,h}+\norm{p_h-\pi_hp}_{1,h}\lesssim \norm{\bm\u-r_h\bm\u}_{L^2(\Omega)}+ h^{\frac{1}{2}}\sum_{f\in\mathcal F_h^\partial(\Gamma_N)}\norm{p-\pi_h p}_{L^2(f)}.
  	\end{equation}
  \end{lemma}	
  \begin{proof}
The stability estimates previously shown for $a_h(\cdot,\cdot)$ and $b_1(\cdot,\cdot)$ with respect to $\norm{\cdot}_{0,h}$ and $\norm{\cdot}_{1,h}$ imply
  \begin{equation}\label{eq:global_inv_m=1}
  \vertiii{\bm\eta_h,s_h}_h \lesssim \sup_{\left(\vv_h,q_h\right)} \frac{\mathcal A_h \left( \left(\bm\eta_h,s_h\right),\left(\bm\vv_h,q_h\right)\right)}{\vertiii{\bm\vv_h,q_h}_h}\qquad\forall\ \left(\bm\eta_h,s_h \right)\in V_h\times Q_h,
  \end{equation}
where
\begin{align*}
\mathcal A_h \left( \left(\bm\eta_h,s_h\right),\left(\vv_h,q_h\right)\right) :=& a_h(\bm\eta_h,\vv_h) + b_1(\vv_h,s_h) + b_1(\bm\eta_h,q_h),\\
\vertiii{\bm\eta_h,s_h}_h^2:=& \norm{\bm\eta_h}^2_{0,h} + \norm{s_h}^2_{1,h}.
\end{align*}
Using~\eqref{eq:global_inv_m=1}, for $(\bm\u_h-r_h\bm\u,p_h-\pi_hp)$ there exists $\left(\vv_h,q_h\right)\in V_h\times Q_h$ such that
\begin{align*}
\norm{\bm\u_h-r_h\bm\u_h}_{0,h}+\norm{p_h-\pi_h p}_{1,h} \le& \sqrt{d} \vertiii{\bm\u_h-r_h\bm\u,p_h-\pi_h p}_h\lesssim 
\frac{\mathcal A_h\left( \left(\bm\u_h-r_h\bm\u,p_h-\pi_h p\right),\left(\vv_h,q_h\right)\right)}{\vertiii{\vv_h,q_h}_h}.
\end{align*}
Hence, we have
\begin{equation}\label{eq1:apriori_m=1}
\begin{aligned}
\mathcal A_h \left( \left(\bm\u_h-r_h\bm\u,p_h-\pi_h p\right),\left(\vv_h,q_h\right)\right) =& \left(\u_h-\u,\vv_h\right)_{L^2(\Omega)}+\left(\u-r_h\u,\vv_h\right)_{L^2(\Omega)}+h^{-1}\langle \left(\u_h-\u\right)\cdot\n,\vv_h\cdot\n\rangle_{\Gamma_N}\\
&+h^{-1}\langle\left(\u-r_h\u\right)\cdot\n,\vv_h\cdot\n\rangle_{\Gamma_N}+b_0(\vv_h,p_h-p)+b_0(\vv_h,p-\pi_hp)\\
&-\langle p_h-p,\vv_h\cdot\n \rangle_{\Gamma_N} - \langle p-\pi_h p,\vv_h\cdot\n \rangle_{\Gamma_N}\\
 &+b_1(\u_h-\u,q_h)+b_1(\u-r_h\u,q_h).
\end{aligned}
\end{equation}	
By construction of $r_h$ and $\pi_h$ we have, respectively,
\begin{equation}\label{eq:orthogonaly_relations}
\begin{aligned}
h^{-1}\langle (\u-r_h\u)\cdot\n,\vv_h\cdot\n \rangle_{\Gamma_N}=0\qquad&\forall\ \vv_h\in V_h,\\
b_1(\u-r_h\u,q_h) = -\sum_{K\in\mathcal T_h} \left(\nabla q_h,\u-r_h\u\right)_K + \sum_{f\in\mathcal F_h^i}\langle [q_h],\left(\u-r_h\u\right)\cdot\n \rangle_f=0\qquad&\forall\ q_h\in Q_h, \\
b_0(\vv_h,p-\pi_h p)=0\qquad&\forall\ \vv_h\in V_h.
\end{aligned}
\end{equation}
By consistency, we have
\begin{align*}
\left(\u_h-\u,\vv_h\right)_{L^2(\Omega)}+h^{-1}\langle \left(\u_h-\u\right)\cdot\n,\vv_h\cdot\n\rangle_{\Gamma_N}+b_0(\vv_h,p_h-p)+\langle p_h-p,\vv_h\cdot\n \rangle_{\Gamma_N}=0\qquad&\forall\ \vv_h\in V_h,\\
b_1(\u_h-\u,q_h) = 0\qquad &\forall\ q_h\in Q_h.
\end{align*}
Hence, in~\eqref{eq1:apriori_m=1} we are left with
\begin{align*}
\mathcal A_h \left( \left(\bm\u_h-r_h\bm\u,p_h-\pi_h p\right),\left(\vv_h,q_h\right)\right) =\left(\u-r_h\u,\vv_h\right)_{L^2(\Omega)}
- \langle p-\pi_h p,\vv_h\cdot\n \rangle_{\Gamma_N} .
\end{align*}
We have
\begin{align*}
\left(\u-r_h\u,\vv_h\right)_{L^2(\Omega)}
- \langle p-\pi_h p,\vv_h\cdot\n \rangle_{\Gamma_N} \le & \norm{\u-r_h\u}_{L^2(\Omega)}\norm{\vv_h}_{L^2(\Omega)}\\
& + \sum_{f\in\mathcal F_h^\partial(\Gamma_N)}h^{\frac{1}{2}}\norm{\left(p-\pi_h p\right)}_{L^2(f)}h^{-\frac{1}{2}}\norm{\vv_h\cdot\n}_{L^2(\Gamma_N)},
\end{align*}
and we can write
\begin{align*}
\norm{\u_h - r_h\u}_{0,h}+\norm{p_h-\pi_h p}_{1,h}\lesssim & \frac{\left(\norm{\u-r_h\u}_{L^2(\Omega)}+h^{\frac{1}{2}}\sum_{f\in\mathcal F_h^\partial(\Gamma_N)}\norm{p-\pi_hp}_{L^2(f)} \right)\vertiii{\vv_h,0}_h}{\vertiii{\vv_h,q_h}_h} \\
\lesssim& \norm{\u-r_h\u}_{L^2(\Omega)}+h^{\frac{1}{2}}\sum_{f\in\mathcal F_h^\partial(\Gamma_N)}\norm{p-\pi_hp}_{L^2(f)}.
\end{align*}
\end{proof}
 \begin{lemma}\label{apriori:lemma1_pert}
	Let $\left(\bm\u^\eps,p^\eps\right)$ be the solution of the perturbed continuous problem~\eqref{eq:prob_cont_bis} and $\left(\bm\u_h,p_h\right)\in V_h\times Q_h$ the one of the discrete problem~\eqref{eq:prob_disc_bis}. Then
	\begin{equation}
	\norm{\bm\u_h - r_h\bm\u^\eps}_{0,h,\eps}+\norm{p_h-\pi_hp^\eps}_{1,h,\eps}\lesssim \norm{\bm\u^\eps-r_h\bm\u^\eps}_{L^2(\Omega)}.
	\end{equation}
\end{lemma}	
\begin{proof}
The stability estimates previously shown for $a_\eps(\cdot,\cdot)$ and $b_0(\cdot,\cdot)$ with respect to $\norm{\cdot}_{0,h,\eps}$ and $\norm{\cdot}_{1,h,\eps}$ imply
	\begin{equation}\label{apriori:eq1_pert}
\vertiii{\bm\eta,s_h}_{h,\eps} \lesssim \sup_{\left(\vv_h,q_h\right)} \frac{\mathcal A_\eps \left( \left(\bm\eta_h,s_h\right),\left(\vv_h,q_h\right)\right)}{\vertiii{\vv_h,q_h}_{h,\eps}}\qquad\forall\ \left(\bm\eta_h,s_h \right)\in V_h\times Q_h,
\end{equation}
where
\begin{align*}
\mathcal A_\eps \left( \left(\bm\eta_h,s_h\right),\left(\vv_h,q_h\right)\right) :=& a_\eps(\bm\eta_h,\vv_h) + b_0(\vv_h,s_h) + b_0(\bm\eta_h,q_h),\\
\vertiii{\bm\eta_h,s_h}^2:=& \norm{\bm\eta_h}^2_{0,h,\eps} + \norm{s_h}^2_{1,h,\eps}.
\end{align*}
Hence, for $(\bm\u_h-r_h\bm\u^\eps,p_h-\pi_hp^\eps)$ there exists $\left(\vv_h,q_h\right)\in V_h\times Q_h$ such that
\begin{align*}
\norm{\bm\u_h-r_h\bm\u^\eps}_{0,h,\eps}+\norm{p_h-\pi_h p^\eps}_{1,h,\eps} \le& \sqrt{d} \vertiii{\bm\u_h-r_h\bm\u^\eps,p_h-\pi_h p^\eps}_{h,\eps}\lesssim 
\frac{\mathcal A_\eps \left( \left(\bm\u_h-r_h\bm\u^\eps,p_h-\pi_h p^\eps\right),\left(\vv_h,q_h\right)\right)}{\vertiii{\vv_h,q_h}_{h,\eps}}.
\end{align*}
Hence, we have
\begin{align*}
\mathcal A_{\eps} \left( \left(\bm\u_h-r_h\bm\u^\eps,p_h-\pi_h p^\eps\right),\left(\vv_h,q_h\right)\right) =& \left(\u_h-\u^\eps,\vv_h\right)_{L^2(\Omega)}+\left(\u^\eps-r_h\u^\eps,\vv_h\right)_{L^2(\Omega)}+\eps^{-1}\langle \left(\u_h-\u^\eps\right)\cdot\n,\vv_h\cdot\n\rangle_{\Gamma_N}\\
&+\eps^{-1}\langle\left(\u^\eps-r_h\u^\eps\right)\cdot\n,\vv_h\cdot\n\rangle_{\Gamma_N}
+b_0(\vv_h,p_h-p^\eps)\\
&+b_0(\vv_h,p^\eps-\pi_hp^\eps) +b_0(\u_h-\u^\eps,q_h)+b_0(\u^\eps-r_h\u^\eps,q_h).
\end{align*}
The following orthogonality relations hold by definition of $r_h$ and $\pi_h$:
\begin{alignat*}{3}
\eps^{-1}\langle\left(\u^\eps-r_h\u^\eps\right)\cdot\n,\vv_h\cdot\n\rangle_{\Gamma_N} = &0\qquad&&\forall\ \vv_h\in V_h,\\
b_0(\vv_h,p^\eps-\pi_h p^\eps)=&0 \qquad&&\forall\ \vv_h\in V_h,\\
b_0(\u^\eps-r_h\u^\eps,q_h) = -\sum_{K\in\mathcal T_h}\left(\nabla q_h,\u^\eps-r_h\u^\eps\right)_K\\
 + \sum_{f\in\mathcal F_h^i}\langle [q_h],\left(\u^\eps-r_h\u^\eps\right)\cdot\n\rangle_f+ \sum_{f\in\mathcal F_h^\partial(\Gamma_N)}\langle q_h,\left(\u^\eps-r_h\u^\eps\right)\cdot\n\rangle_f=&0\qquad&&\forall\ q_h\in Q_h.
\end{alignat*}
Moreover, by consistency, we have
\begin{align*}
\left(\u_h-\u^\eps,\vv_h\right)_{L^2(\Omega)}+\eps \langle \left( \u_h-\u^\eps\right)\cdot\n,\vv_h\cdot\n\rangle_{\Gamma_N} + b_0(\vv_h,p_h-p^\eps) =   0\qquad&\forall\ \vv_h\in V_h,\\
b_0(\u_h-\u^\eps,q_h) = 0\qquad&\forall\ q_h\in Q_h.
\end{align*}
Hence,
\begin{align*}
\mathcal A_\eps \left( \left(\bm\u_h-r_h\bm\u^\eps,p_h-\pi_h p^\eps\right),\left(\bm\tau_h,q_h\right)\right)=  \left(\u^\eps-r_h\u^\eps,\vv_h\right)_{L^2(\Omega)},
\end{align*}
and we can write
\begin{align*}
\norm{\u_h - r_h\u^\eps}_{0,h,\eps}+\norm{p_h-\pi_h p^\eps}_{1,h,\eps}\lesssim  \frac{\norm{\u^\eps-r_h\u^\eps}_{L^2(\Omega)}\vertiii{\vv_h,0}_{h,\eps}}{\vertiii{\vv_h,q_h}_{h,\eps}} 
\lesssim \norm{\u^\eps-r_h\u^\eps}_{L^2(\Omega)}.
\end{align*}
\end{proof}	

	\begin{proposition}\label{apriori:prop1_m=1}
	Let $\left(\bm\u,p\right) \in \bm H^{r+1}(\Omega)\times H^{t+1}(\Omega)$ and $s:=\min\{r,t,k\}$ be the solution of~\eqref{eq:prob_cont} and $\left(\u_h,p_h\right)\in V_h\times Q_h$ the one to~\eqref{eq:prob_disc} with $m=1$. There exists $C>0$ such that
	\begin{equation}\label{apriori:eq5_m=1}
	\norm{\bm\u_h-r_h\bm\u}_{0,h} + \norm{p_h-\pi_h p}_{1,h}\le C h^{s+1} \left( \norm{\u}_{H^{r+1}(\Omega)} +\norm{p}_{H^{t+1}(\Omega)}\right).
	\end{equation}	
\end{proposition}	
\begin{proof}
	By Lemma~\ref{apriori:lemma1_m=1}, a multiplicative trace inequality for Sobolev functions and standard approximation results for the $L^2$-projection
	\begin{align*}
\norm{\bm\u_h-r_h\bm\u}_{0,h} + \norm{p_h-\pi_h p}_{1,h} \lesssim & \norm{\u-r_h\u}_{L^2(\Omega)} + \sum_{f\in\mathcal F_h^\partial(\Gamma_N)}h^{\frac{1}{2}}\norm{p-\pi_h p}_{L^2(f)} \\
\lesssim & \norm{\u-r_h\u}_{L^2(\Omega)} + \sum_{K\in\mathcal T_h}h^{\frac{1}{2}}\norm{p-\pi_h p}^{\frac{1}{2}}_{L^2(K)} \norm{\nabla\left(p-\pi_h p\right)}^{\frac{1}{2}}_{L^2(K)}  \\
\lesssim &  \norm{\u-r_h\u}_{L^2(\Omega)} + h^{\frac{1}{2}}h^{\frac{t+1}{2}}\norm{p}^{\frac{1}{2}}_{H^{t+1}(\Omega)}h^{\frac{t}{2}} \norm{p}^{\frac{1}{2}}_{H^{t+1}(\Omega)}\\
= &\norm{\u-r_h\u}_{L^2(\Omega)} + h^{t+1}\norm{p}_{H^{t+1}(\Omega)}.
	\end{align*}
By using Bramble-Hilbert/Deny-Lions Lemma~\cite{qvalli}, we get
\begin{align*}
\norm{\bm\u_h-r_h\bm\u}_{0,h} + \norm{p_h-\pi_h p}_{1,h} \lesssim h^{r+1} \norm{\u}_{H^{r+1}(\Omega)} + h^{t+1}\norm{p}_{H^{t+1}(\Omega)},
\end{align*}
with $0\le t\le k$ and $0\le r \le k$.
\end{proof}


	\begin{proposition}\label{apriori:prop1_pert}
	Let $\left(\bm\u^\eps,p^\eps\right) \in \bm H^{r+1}(\Omega)\times H^{t+1}(\Omega)$ and $s:=\min\{r,k\}$ be the solution of the perturbed continuous problem~\eqref{eq:prob_cont_bis} and $\left(\u_h,p_h\right)\in V_h\times Q_h$ the one to~\eqref{eq:prob_disc_bis}. There exists $C>0$ such that
	\begin{equation}\label{apriori:eq5_pert}
	\norm{\bm\u_h-r_h\bm\u^\eps}_{0,h,\eps} + \norm{p_h-\pi_h p^\eps}_{1,h,\eps}\le C h^{s+1} \norm{\u^\eps}_{H^{s+1}(\Omega)}.
	\end{equation}	
\end{proposition}	
\begin{proof}
	By Lemma~\ref{apriori:lemma1_pert} and Bramble-Hilbert/Deny-Lions Lemma~\cite{qvalli}, we get
	\begin{align*}
	\norm{\bm\u_h-r_h\bm\u^\eps}_{0,h,\eps} + \norm{p_h-\pi_h p^\eps}_{1,h,\eps} \lesssim h^{s+1} \norm{\u^\eps}_{H^{s+1}(\Omega)}.
	\end{align*}
\end{proof}	

	\begin{remark}
		Let us remark that the quantities $\norm{\u_h-r_h\u_h}_{0,h}$, $\norm{p_h-\pi_h p}_{1,h}$ in~\eqref{apriori:eq5_m=1} and $\norm{p_h-\pi_h p^\eps}_{1,h,\eps}$, $\norm{\u_h-r_h\u^\eps}_{0,h,\eps}$  in~\eqref{apriori:eq5_pert}, respectively, are super convergent.
	\end{remark}
	\begin{theorem}\label{apriori:theorem1_m=1}
	Let $\left(\bm\u,p\right) \in  \bm H^{r+1}(\Omega)\times H^{t+1}(\Omega)$ be the solution to~\eqref{eq:prob_cont} and $\left(\bm\u_h,p_h\right) \in  V_h\times Q_h$ the one to~\eqref{eq:prob_disc} with $m=1$. Then there exists $C>0$ such that, for $s:=\min\{r,t,k\}$,
	\begin{align*}
	\norm{\u-\u_h}_{L^2(\Omega)}\le C &  h^{s+1}\left( \norm{\bm\u}_{H^{r+1}(\Omega)} + \norm{p}_{H^{t+1}(\Omega)}\right).
	\end{align*}
\end{theorem}
\begin{proof}
	Let us proceed by triangular inequality.
	\begin{align*}
	\norm{\u-\u_h}_{L^2(\Omega)} \le \norm{\u-r_h\u}_{L^2(\Omega)}+ \norm{r_h\u-\u_h}_{L^2(\Omega)}.
	\end{align*}
	The first and the second terms in the rhs scale as $\mathcal O(h^{r+1})$ and $\mathcal O (h^{s+1})$, respectively, because of Bramble-Hilbert/Deny-Lions Lemma~\cite{qvalli} and Proposition~\ref{apriori:prop1_m=1}.
\end{proof}
	\begin{lemma}\label{apriori:lemma2_pert}
	Let $\left(\bm\u^\eps,p^\eps\right) \in  \bm H^{r+1}(\Omega)\times H^{t+1}(\Omega)$ be the solution to the perturbed continuous problem~\eqref{eq:prob_cont_bis} and $\left(\bm\u_h,p_h\right) \in  V_h\times Q_h$ the one to~\eqref{eq:prob_disc_bis}. Then there exists $C>0$ such that, for $s:=\min\{r,k\}$,
	\begin{align*}
	\norm{\u^\eps-\u_h}_{L^2(\Omega)}\le C   h^{s+1} \norm{\bm\u^\eps}_{H^{s+1}(\Omega)}.
	\end{align*}
\end{lemma}
\begin{proof}
	Let us proceed by triangular inequality.
	\begin{align*}
	\norm{\u^\eps-\u_h}_{L^2(\Omega)} \le \norm{\u^\eps-r_h\u^\eps}_{L^2(\Omega)}+ \norm{r_h\u^\eps-\u_h}_{L^2(\Omega)}.
	\end{align*}
	The first and the second terms in the rhs scale as $\mathcal O(h^{s+1})$, respectively, because of Bramble-Hilbert/Deny-Lions Lemma~\cite{qvalli} and Proposition~\ref{apriori:prop1_pert}.
\end{proof}
	\begin{theorem}\label{apriori:theorem1_pert}
	Let $\left(\bm\u,p\right) \in  \bm H^{2}(\Omega)\times H^{t+1}(\Omega)$ be the solution to the continuous~\eqref{eq:prob_cont} and $\left(\bm\u_h,p_h\right) \in  V_h\times Q_h$ the one to~\eqref{eq:prob_disc_bis}. Assume $\Omega$ to be a convex with a  Lipschitz polygonal boundary $\Gamma$, $\f\in \bm H^{1}(\Omega)$ and $u_N=0$. Then, there exists $C>0$ such that
	\begin{align*}
	\norm{\u-\u_h}_{L^2(\Omega)} \le C h \left(\norm{ \f}_{H^1(\Omega)} + \norm{g}_{L^2(\Omega)}+ \norm{p_D}_{H^\frac{1}{2}(\Gamma_D)} \right).
	\end{align*}
\end{theorem}
\begin{proof}
	Let us proceed by triangular inequality.
	\begin{align*}
	\norm{\u-\u_h}_{L^2(\Omega)} \le & \norm{\u-\u^\eps}_{L^2(\Omega)}+ \norm{\u^\eps-\u_h}_{L^2(\Omega)} \lesssim \eps \left(\norm{\dive \f}_{L^2(\Omega)} + \norm{g}_{L^2(\Omega)}+ \norm{p_D}_{H^\frac{1}{2}(\Gamma_D)}	\right)+ h \norm{\u^\eps}_{H^{1}(\Omega)}\\
	 \lesssim  &  \eps \left(\norm{\dive \f}_{L^2(\Omega)} + \norm{g}_{L^2(\Omega)}+ \norm{p_D}_{H^\frac{1}{2}(\Gamma_D)}	\right)+ h\left( \norm{\f}_{H^{1}(\Omega)} + \norm{g}_{L^{2}(\Omega)}  \right).
	\end{align*} 
We used Lemma~\ref{apriori:lemma2_pert}, Proposition~\ref{prop:conv_eps}, and finally Proposition~\ref{prop:stability_eps_simple} combined with Remark~\ref{remark:bd}. Finally, let us choose we just choose $\eps=h$.
\end{proof}

\begin{remark}
We observe that for both formulations,~\eqref{eq:prob_disc} and~\eqref{eq:prob_disc_bis}, all dimensionless parameters have been set for simplicity to $1$, unlike for the standard Nitsche method for the Poisson problem~\cite{STENBERG1995139}, where the dimensionless parameter needs to be taken large enough.
\end{remark}

%% file: Sections/numerical_experiments.tex
\section{Numerical examples}
\subsection{Convergence results}
In this first set of numerical examples we verify that the optimal a priori error estimates of Theorems~\ref{apriori:theorem1_m=1},~\ref{apriori:theorem1_pert}. We also check that the result of Theorem~\ref{apriori:theorem1_m=1} holds in the non-symmetric case $m=0$, as already mentioned in section~\ref{sec:fe_disc}. Moreover, we study the $L^2$ error of the pressure field, for which optimal convergence is observed in general and super convergence in the case of the lowest order Raviart-Thomas element and triangular meshes. 

Although Theorem~\ref{apriori:theorem1_pert} guarantees us optimal a priori error estimates for the discretization~\eqref{eq:prob_disc_bis} only with the lowest order Raviart-Thomas element, numerical results show that we have optimal convergence rates also for higher orders.
\subsubsection{Unit square with triangular meshes}\label{unit_square_triangles}
We approximate the Darcy problem in the unit square $\Omega=\left(0,1\right)^2$ using a family of triangular meshes, with weakly enforced Neumann boundary conditions on the whole boundary, using as manufactured solutions
\begin{align*}
\u_{ex}=
\begin{pmatrix}
x\sin(x)\sin(y)\\
\sin(x)\cos(y) + x\cos(x)\cos(y)	
\end{pmatrix},\qquad
p_{ex} = x^3 y - 0.125.	
\end{align*}
Note that $\u_{ex}$ is divergence-free. The numerical results are in Figures~\ref{num_exp:fig1_ff},~\ref{num_exp:fig2_ff} and~\ref{num_exp:fig3_ff}.
\begin{figure}[!ht]
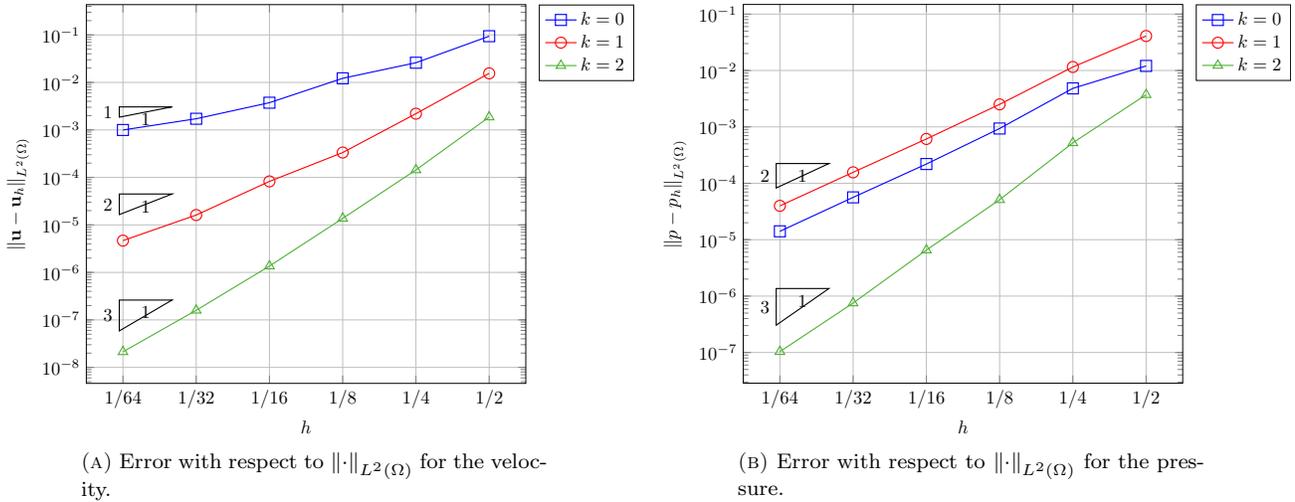

	\centering
	\subfloat[][Error with respect to $\norm{\cdot}_{L^2(\Omega)}$ for the velocity.]
	{
		\includestandalone[width=0.5\textwidth,keepaspectratio=true]{Figures/error_vel_l2_unit_square_ff}
	}
	\subfloat[][Error with respect to $\norm{\cdot}_{L^2(\Omega)}$ for the pressure.]
	{
		\includestandalone[width=0.5\textwidth,keepaspectratio=true]{Figures/error_press_unit_square_ff}
	}
	\caption{Convergence errors in the ``unit square'' using~\eqref{eq:prob_disc} with $m=1$ with triangular meshes.}\label{num_exp:fig1_ff}
\end{figure}
\begin{figure}[!ht]
	\centering
	\subfloat[][Error with respect to $\norm{\cdot}_{L^2(\Omega)}$ for the velocity.]
	{
		\includestandalone[width=0.5\textwidth,keepaspectratio=true]{Figures/error_vel_l2_unit_square_ff_m=0}
	}
	\subfloat[][Error with respect to $\norm{\cdot}_{L^2(\Omega)}$ for the pressure.]
	{
		\includestandalone[width=0.5\textwidth,keepaspectratio=true]{Figures/error_press_unit_square_ff_m=0}
	}
	\caption{Convergence errors in the ``unit square'' using~\eqref{eq:prob_disc} with $m=0$ with triangular meshes.}\label{num_exp:fig2_ff}
\end{figure}
\begin{figure}[!ht]
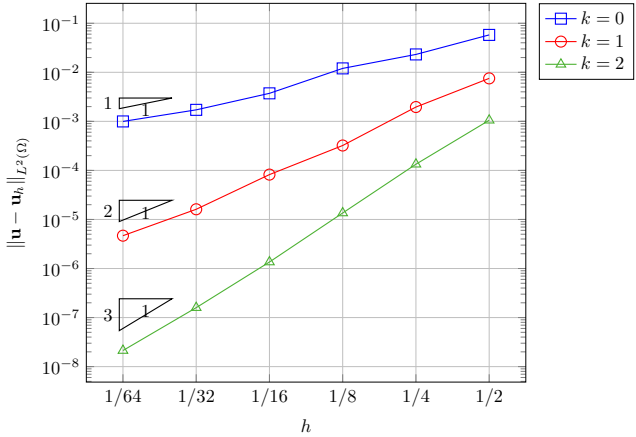
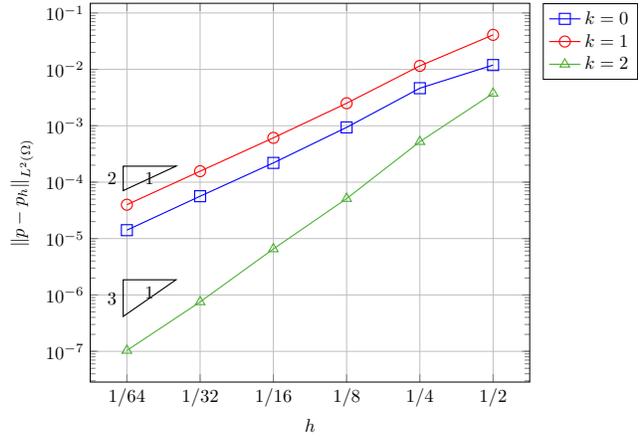

	\centering
	\subfloat[][Error with respect to $\norm{\cdot}_{L^2(\Omega)}$ for the velocity.]
	{
		\includestandalone[width=0.5\textwidth,keepaspectratio=true]{Figures/error_vel_l2_unit_square_ff_bis}
	}
	\subfloat[][Error with respect to $\norm{\cdot}_{L^2(\Omega)}$ for the pressure.]
	{
		\includestandalone[width=0.5\textwidth,keepaspectratio=true]{Figures/error_press_unit_square_ff_bis}
	}
	\caption{Convergence errors in the ``unit square'' using~\eqref{eq:prob_disc_bis} with triangular meshes.}\label{num_exp:fig3_ff}
\end{figure}

\subsubsection{Unit circle with triangular meshes}\label{unit_circle}
Now, we consider the unit circle $\Omega=\{(x,y)\in R^2:x^2+y^2\le 1\}$ which is meshed using triangles. We weakly impose the essential boundary conditions on the boundary and consider the following reference solutions:
\begin{align*}
	\u_ex=
	\begin{pmatrix}
		\frac{1}{10}e^x\sin(xy)\\
		x^4+y^2	
	\end{pmatrix},\qquad
	p_{ex} = 	x^3\cos(x)+y^2\sin(x).
\end{align*}
This time $\dive \u_{ex}=2y+\frac{1}{10}\left( e^x\sin(xy)+ye^x\cos(xy)\right)$. See Figures~\ref{num_exp:fig4_ff},~\ref{num_exp:fig6_ff} and~\ref{num_exp:fig8_ff}.
\begin{figure}[!ht]
	\centering
	\subfloat[][Error with respect to $\norm{\cdot}_{L^2(\Omega)}$.]
	{
		\includestandalone[width=0.5\textwidth,keepaspectratio=true]{Figures/error_vel_l2_unit_circle_ff}
	}
	\subfloat[][Error with respect to $\norm{\cdot}_{L^2(\Omega)}$ for the pressure.]
	{
		\includestandalone[width=0.5\textwidth,keepaspectratio=true]{Figures/error_press_unit_circle_ff}
	}
	\caption{Convergence errors in the``unit circle'' using~\eqref{eq:prob_disc} with $m=1$ with triangular meshes.}\label{num_exp:fig4_ff}
\end{figure}
\begin{figure}[!ht]
	\centering
	\subfloat[][Error with respect to $\norm{\cdot}_{L^2(\Omega)}$.]
	{
		\includestandalone[width=0.5\textwidth,keepaspectratio=true]{Figures/error_vel_l2_unit_circle_ff_m=0}
	}
	\subfloat[][Error with respect to $\norm{\cdot}_{L^2(\Omega)}$ for the pressure.]
	{
		\includestandalone[width=0.5\textwidth,keepaspectratio=true]{Figures/error_press_unit_circle_ff_m=0}
	}
	\caption{Convergence errors in the ``unit circle'' using~\eqref{eq:prob_disc} with $m=0$ with triangular meshes.}\label{num_exp:fig6_ff}
\end{figure}
\begin{figure}[!ht]
	\centering
	\subfloat[][Error with respect to $\norm{\cdot}_{L^2(\Omega)}$.]
	{
		\includestandalone[width=0.5\textwidth,keepaspectratio=true]{Figures/error_vel_l2_unit_circle_ff_bis}
	}
	\subfloat[][Error with respect to $\norm{\cdot}_{L^2(\Omega)}$ for the pressure.]
	{
		\includestandalone[width=0.5\textwidth,keepaspectratio=true]{Figures/error_press_unit_circle_ff_bis}
	}
	\caption{Convergence errors in the ``unit circle'' using~\eqref{eq:prob_disc_bis} with triangular meshes.}\label{num_exp:fig8_ff}
\end{figure}

\subsubsection{Unit square with quadrilateral meshes}\label{unit_square_quads}
Let us consider the unit square $\Omega=\left(0,1\right)^2$ meshed using quadrilaterals. We impose natural boundary conditions on $\{(x,y): 0\le x\le 1, y=0\}$ and essential boundary conditions everywhere else in a weak sense. The reference solutions are:
\begin{align*}
\u_{ex} =
\begin{pmatrix}
\cos(x)\operatorname{cosh}(y\\
\sin(x)\operatorname{cosh}(y))	
\end{pmatrix},\qquad
p_{ex}= -\sin(x)\sinh (y)- \left(\cos(1)-1 \right)\left(\cosh(1)-1\right).
\end{align*}
We have $\dive\u_{ex}=0$. For the numerical results we refer to Figures~\ref{num_exp:fig1},~\ref{num_exp:fig1_m=0} and~\ref{num_exp:fig1_bis}.
\begin{figure}[!ht]
	\centering
	\subfloat[][Error with respect to $\norm{\cdot}_{L^2(\Omega)}$ for the velocity.]
	{
		\includestandalone[width=0.5\textwidth,keepaspectratio=true]{Figures/error_vel_l2_unit_square}
	}
	\subfloat[][Error with respect to $\norm{\cdot}_{L^2(\Omega)}$ for the pressure.]
	{
		\includestandalone[width=0.5\textwidth,keepaspectratio=true]{Figures/error_press_unit_square}
	}
	\caption{Convergence errors in the ``unit square'' using~\eqref{eq:prob_disc} with $m=1$ with quadrilateral meshes.}\label{num_exp:fig1}
\end{figure}
\begin{figure}[!ht]
	\centering
	\subfloat[][Error with respect to $\norm{\cdot}_{L^2(\Omega)}$ for the velocity.]
	{
		\includestandalone[width=0.5\textwidth,keepaspectratio=true]{Figures/error_vel_l2_unit_square_m=0}
	}
	\subfloat[][Error with respect to $\norm{\cdot}_{L^2(\Omega)}$ for the pressure.]
	{
		\includestandalone[width=0.5\textwidth,keepaspectratio=true]{Figures/error_press_unit_square_m=0}
	}
	\caption{Convergence errors in the ``unit square'' using~\eqref{eq:prob_disc} with $m=0$ with quadrilateral meshes.}\label{num_exp:fig1_m=0}
\end{figure}
\begin{figure}[!ht]
	\centering
	\subfloat[][Error with respect to $\norm{\cdot}_{L^2(\Omega)}$ for the velocity.]
	{
		\includestandalone[width=0.5\textwidth,keepaspectratio=true]{Figures/error_vel_l2_unit_square_bis}
	}
	\subfloat[][Error with respect to $\norm{\cdot}_{L^2(\Omega)}$ for the pressure.]
	{
		\includestandalone[width=0.5\textwidth,keepaspectratio=true]{Figures/error_press_unit_square_bis}
	}
	\caption{Convergence errors in the ``unit square'' using~\eqref{eq:prob_disc_bis} with quadrilateral meshes.}\label{num_exp:fig1_bis}
\end{figure}

\subsubsection{Quarter of annulus with quadrilateral isoparametric elements}\label{quarter_annulus}
Let us consider the quarter of annulus centered in the origin with inner and outer radii, respectively, $r=1$ and $R=2$, discretized using quadrilateral isoparametric elements~\cite{claes}. We impose natural boundary conditions on the straight edges $\{(x,y):1\le x\le 2, y=0\}$ and $\{(x,y):x=0, 1\le y\le 2\}$ and weak essential boundary conditions on the curved ones. The manufactured solutions are:
\begin{align*}
	\u_{ex} =
	\begin{pmatrix}
		 -xy^2\\
		-x^2y-\frac{3}{2}y^2
	\end{pmatrix},\qquad
	p_{ex}= \frac{1}{2}\left( x^2y^2 + y^3\right),
\end{align*}
with $\dive \u_{ex}= -x^2 - y^2 - 3y$. See Figures~\ref{num_exp:fig3},~\ref{num_exp:fig3_m=0} and~\ref{num_exp:fig3_bis}.
\begin{figure}[!ht]
	\centering
	\subfloat[][Error with respect to $\norm{\cdot}_{L^2(\Omega)}$.]
	{
		\includestandalone[width=0.5\textwidth,keepaspectratio=true]{Figures/error_vel_l2_quarter_annulus}
	}
	\subfloat[][Error with respect to $\norm{\cdot}_{L^2(\Omega)}$ for the pressure.]
	{
		\includestandalone[width=0.5\textwidth,keepaspectratio=true]{Figures/error_press_quarter_annulus}
	}
	\caption{Convergence errors in the ``quarter of annulus'' using~\eqref{eq:prob_disc} with $m=1$ with isoparametric quadrilateral elements.}\label{num_exp:fig3}
\end{figure}
\begin{figure}[!ht]
	\centering
	\subfloat[][Error with respect to $\norm{\cdot}_{L^2(\Omega)}$.]
	{
		\includestandalone[width=0.5\textwidth,keepaspectratio=true]{Figures/error_vel_l2_quarter_annulus_m=0}
	}
	\subfloat[][Error with respect to $\norm{\cdot}_{L^2(\Omega)}$ for the pressure.]
	{
		\includestandalone[width=0.5\textwidth,keepaspectratio=true]{Figures/error_press_quarter_annulus_m=0}
	}
	\caption{Convergence errors in the ``quarter of annulus'' using~\eqref{eq:prob_disc} with $m=0$ with isoparametric quadrilateral elements.}\label{num_exp:fig3_m=0}
\end{figure}
\begin{figure}[!ht]
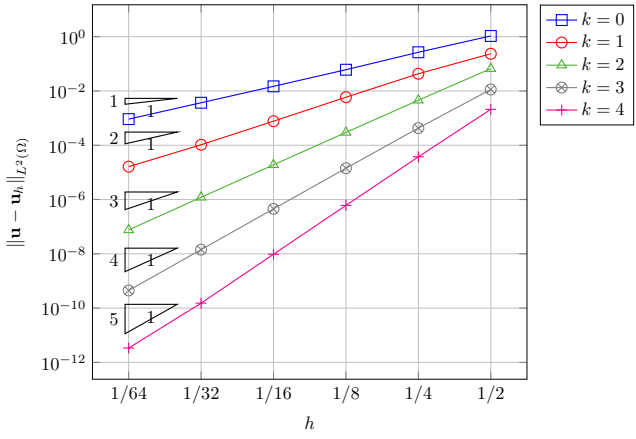
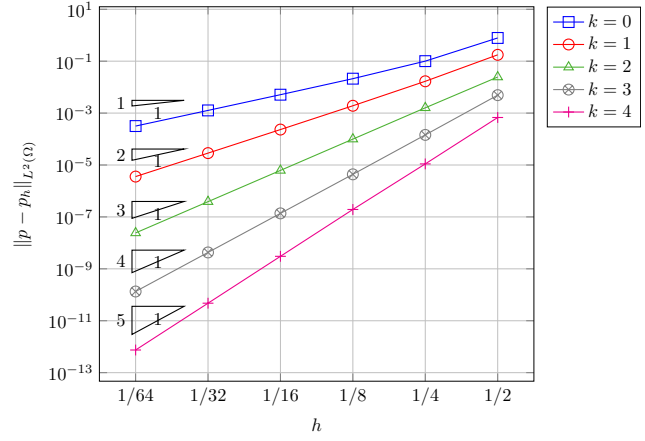

	\centering
	\subfloat[][Error with respect to $\norm{\cdot}_{L^2(\Omega)}$.]
	{
		\includestandalone[width=0.5\textwidth,keepaspectratio=true]{Figures/error_vel_l2_quarter_annulus_bis}
	}
	\subfloat[][Error with respect to $\norm{\cdot}_{L^2(\Omega)}$ for the pressure.]
	{
		\includestandalone[width=0.5\textwidth,keepaspectratio=true]{Figures/error_press_quarter_annulus_bis}
	}
	\caption{Convergence errors in the ``quarter of annulus'' using~\eqref{eq:prob_disc_bis} with isoparametric quadrilateral elements.}\label{num_exp:fig3_bis}
\end{figure}
\FloatBarrier
\subsection{A remark about the condition numbers}\label{conditioning}
Proceeding as in~\cite{ern_guermond_cond} it would be possible to prove that the $\ell^2$-condition number of the stiffness matrix arising from the discretizations~\eqref{eq:prob_disc}, for both $m\in\{0,1\}$,	scales as $h^{-2}$, as Figures~\ref{num_exp:fig6} and~\ref{num_exp:fig4} confirm. The penalty parameter for the weak imposition of the Neumann boundary conditions is the responsible of the deterioration of the conditioning with respect to the standard mixed finite element discretization of the Poisson problem, for which the condition number scales as $h^{-1}$. An even worse situation occurs when formulation~\eqref{eq:prob_disc_bis} is employed. In this case the condition number scales as $h^{-\left(s+2 \right)}$, $s=\min\{r,k\}$, $r$ being the Sobolev regularity of the exact solution for the pressure field and $k$ the polynomial degree of the Raviart-Thomas discretization, as confirmed by Figures~\ref{num_exp:fig7} and~\ref{num_exp:fig5}.

In all numerical experiments, we do not detect any particular sensitivity of the convergence of the error of the velocities with respect to $\gamma$ in the case of method~\eqref{eq:prob_disc}. On the other hand, it is a different matter altogether as far as the formulation~\eqref{eq:prob_disc_bis} is concerned: this time we can realize the influence of $\gamma$ on the approximation power of the method.

\begin{figure}[!ht]
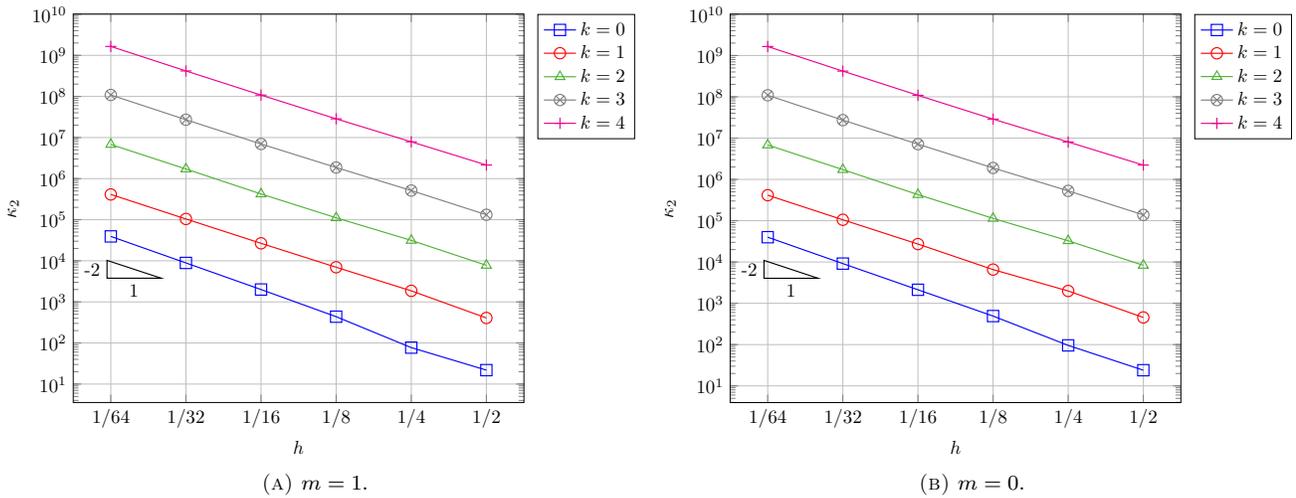

	\centering
	\subfloat[][$m=1$.]
	{
		\includestandalone[width=0.5\textwidth,keepaspectratio=true]{Figures/cond_unit_square}
	}
	\subfloat[][$m=0$.]
	{
		\includestandalone[width=0.5\textwidth,keepaspectratio=true]{Figures/cond_unit_square_m=0}
	}
	\caption{Condition numbers in the ``unit square'' using~\eqref{eq:prob_disc} -  GeoPDEs.}\label{num_exp:fig6}
\end{figure}
\begin{figure}[!ht]
	\centering
	\includestandalone[width=0.5\textwidth,keepaspectratio=true]{Figures/cond_unit_square_bis}
	\caption{Condition numbers in the ``unit square'' using~\eqref{eq:prob_disc_bis} -  GeoPDEs.}\label{num_exp:fig7}
\end{figure}

\begin{figure}[!ht]
	\centering
	\subfloat[][$m=1$.]
	{
		\includestandalone[width=0.5\textwidth,keepaspectratio=true]{Figures/cond_quarter_annulus}
	}
	\subfloat[][$m=0$.]
	{
		\includestandalone[width=0.5\textwidth,keepaspectratio=true]{Figures/cond_quarter_annulus_m=0}
	}
	\caption{Condition numbers in the ``quarter of annulus'' using~\eqref{eq:prob_disc} -  GeoPDEs.}\label{num_exp:fig4}
\end{figure}
\begin{figure}[!ht]
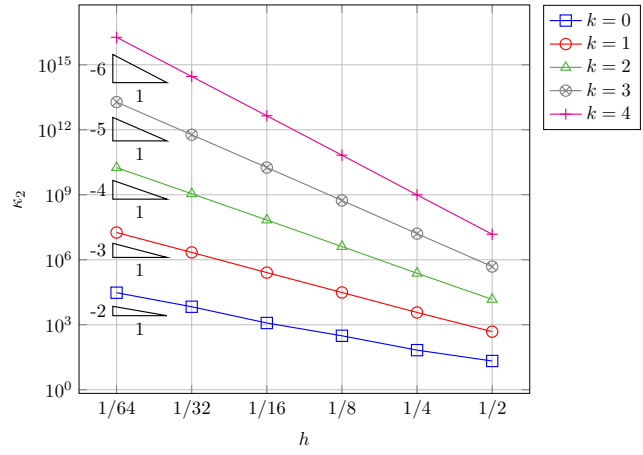

	\centering
\includestandalone[width=0.5\textwidth,keepaspectratio=true]{Figures/cond_quarter_annulus_bis}
	\caption{Condition numbers in the ``quarter of annulus'' using~\eqref{eq:prob_disc_bis} -  GeoPDEs.}\label{num_exp:fig5}
\end{figure}

\FloatBarrier
\subsection{The optimality of the penalty parameter}
We want to analyze the optimality of the penalty parameter, denoted through this subsection as $\gamma$, for both numerical schemes. We consider the Raviart-Thomas element of order $k=1$ and compare the numerical results for the $L^2$-error of the velocity field with respect to different powers of the mesh-size as penalty parameter. The first set of numerical experiences is performed using triangular meshes, then we move to quadrilaterals.

To obtain Figures~\ref{num_exp:fig10} and~\ref{num_exp:fig10_bis} the same setting of subsection~\ref{unit_square_triangles} is employed. Then, in Figures~\ref{num_exp:fig11} and~\ref{num_exp:fig11_bis}, we move to the configuration of subsection \ref{unit_circle}. Finally, in Figures~\ref{num_exp:fig8},~\ref{num_exp:fig8_bis} and~\ref{num_exp:fig9},~\ref{num_exp:fig9_bis} we use, respectively, the settings of subsections~\ref{unit_square_quads} and~\ref{quarter_annulus}.

\begin{figure}[!ht]
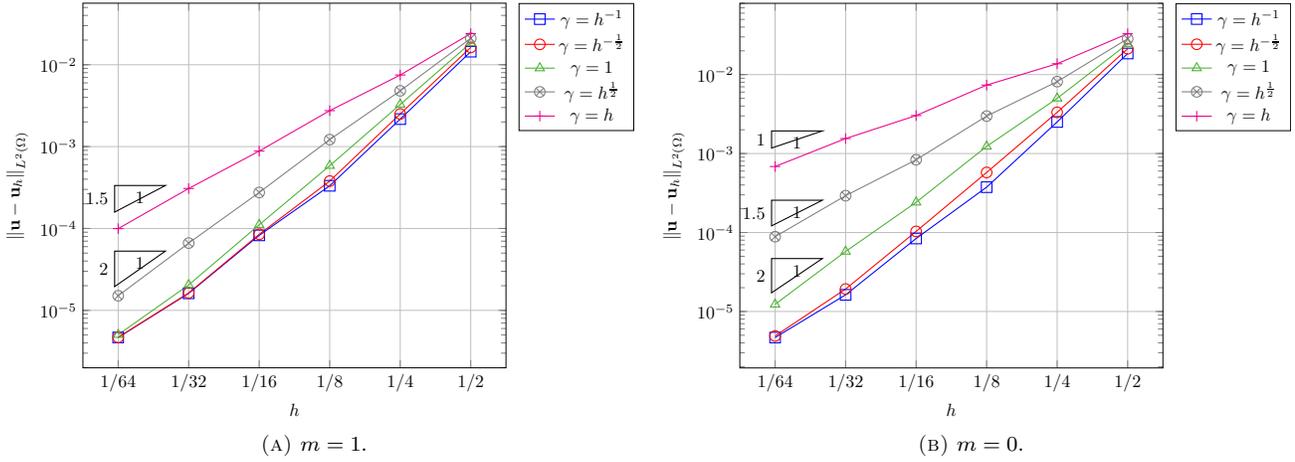

	\centering
	\subfloat[][$m=1$.]
	{
		\includestandalone[width=0.5\textwidth,keepaspectratio=true]{Figures/error_unit_square_compare_penalty_m=1_ff}
	}
	\subfloat[][$m=0$.]
	{
		\includestandalone[width=0.5\textwidth,keepaspectratio=true]{Figures/error_unit_square_compare_penalty_m=0_ff}
	}
	\caption{Compare $L^2$-errors for the velocity in the ``unit square'' using~\eqref{eq:prob_disc} with respect to different values of the penalty parameter $\gamma$ with triangular elements.}\label{num_exp:fig10}
\end{figure}
\begin{figure}[!ht]
	\centering
	\includestandalone[width=0.5\textwidth,keepaspectratio=true]{Figures/error_unit_square_compare_penalty_bis_ff}
	\caption{Compare $L^2$-errors for the velocity in the ``unit square'' using~\eqref{eq:prob_disc_bis} with respect to different values of the penalty parameter $\gamma$ with triangular elements.}\label{num_exp:fig10_bis}
\end{figure}
\FloatBarrier

\begin{figure}[!ht]
	\centering
	\subfloat[][$m=1$.]
	{
		\includestandalone[width=0.5\textwidth,keepaspectratio=true]{Figures/error_unit_circle_compare_penalty_m=1_ff}
	}
	\subfloat[][$m=0$.]
	{
		\includestandalone[width=0.5\textwidth,keepaspectratio=true]{Figures/error_unit_circle_compare_penalty_m=0_ff}
	}
	\caption{Compare $L^2$-errors for the velocity in the ``unit circle'' using~\eqref{eq:prob_disc} with respect to different values of the penalty parameter $\gamma$ with triangular elements.}\label{num_exp:fig11}
\end{figure}
\begin{figure}[!ht]
	\centering
	\includestandalone[width=0.5\textwidth,keepaspectratio=true]{Figures/error_unit_circle_compare_penalty_bis_ff}
	\caption{Compare $L^2$-errors for the velocity in the ``unit circle'' using~\eqref{eq:prob_disc_bis} with respect to different values of the penalty parameter $\gamma$ with triangular elements.}\label{num_exp:fig11_bis}
\end{figure}
\FloatBarrier

\begin{figure}[!ht]
	\centering
	\subfloat[][$m=1$.]
	{
		\includestandalone[width=0.5\textwidth,keepaspectratio=true]{Figures/error_unit_square_compare_penalty_m=1}
	}
	\subfloat[][$m=0$.]
	{
		\includestandalone[width=0.5\textwidth,keepaspectratio=true]{Figures/error_unit_square_compare_penalty_m=0}
	}
	\caption{Compare $L^2$-errors for the velocity in the ``unit square'' using~\eqref{eq:prob_disc} with respect to different values of the penalty parameter $\gamma$ with triangular elements.}\label{num_exp:fig8}
\end{figure}
\begin{figure}[!ht]
	\centering
	\includestandalone[width=0.5\textwidth,keepaspectratio=true]{Figures/error_unit_square_compare_penalty_bis}
	\caption{Compare $L^2$-errors for the velocity in the ``unit square'' using~\eqref{eq:prob_disc_bis} with respect to different values of the penalty parameter $\gamma$ with quadrilateral elements.}\label{num_exp:fig8_bis}
\end{figure}
\FloatBarrier
\begin{figure}[!ht]
	\centering
	\subfloat[][$m=1$.]
	{
		\includestandalone[width=0.5\textwidth,keepaspectratio=true]{Figures/error_quarter_annulus_compare_penalty_m=1}
	}
	\subfloat[][$m=0$.]
	{
		\includestandalone[width=0.5\textwidth,keepaspectratio=true]{Figures/error_quarter_annulus_compare_penalty_m=0}
	}
	\caption{Compare $L^2$-errors for the velocity in the ``quarter of annulus'' using~\eqref{eq:prob_disc} with respect to different values of the penalty parameter $\gamma$ with quadrilateral elements.}\label{num_exp:fig9}
\end{figure}
\begin{figure}[!ht]
	\centering
	\includestandalone[width=0.5\textwidth,keepaspectratio=true]{Figures/error_quarter_annulus_compare_penalty_bis}
	\caption{Compare $L^2$-errors for the velocity in the ``quarter of annulus'' using~\eqref{eq:prob_disc_bis} with respect to different values of the penalty parameter $\gamma$ with quadrilateral elements.}\label{num_exp:fig9_bis}
\end{figure}